\documentclass[10pt]{amsart}
\usepackage{mathrsfs} % for caligraphic L for Lagrangian
\usepackage[all]{xy} % for commutative diagram
\usepackage{graphicx} % for pdfTeX

\newtheorem{thm}{Theorem}[section]
\newtheorem{prop}[thm]{Proposition}

\newtheorem{lem}[thm]{Lemma}

\theoremstyle{definition}

\numberwithin{equation}{section}
\newtheorem{rem}[thm]{Remark}
\newtheorem{ex}[thm]{Example}
\newtheorem{defn}[thm]{Definition}

\newtheorem{ass}[thm]{Assumption}

\begin{document}
\bibliographystyle{amsalpha}

\title[Hamiltonian and Lagrangian formalisms of mutations]
{
Hamiltonian and Lagrangian formalisms\\ of mutations in cluster algebras\\
and application to dilogarithm identities
}

\author{Michael Gekhtman}
\address{\noindent 
Department of Mathematics, University of Notre Dame, Notre Dame, IN 46556,
USA}
\email{mgekhtma@nd.edu}
\author{Tomoki Nakanishi}
\address{\noindent Graduate School of Mathematics, Nagoya University, 
Chikusa-ku, Nagoya,
464-8604, Japan}
\email{nakanisi@math.nagoya-u.ac.jp}
\author{Dylan Rupel}
\address{\noindent 
Department of Mathematics, University of Notre Dame, Notre Dame, IN 46556,
USA}
\email{drupel@nd.edu}

%\subjclass[2010]{Primary 13F60, Secondary 33E20}
%\keywords{dilogarithm, quantum dilogarithm, cluster algebra}
%\thanks{This work was partially supported by JSPS KAKENHI Grant Number 16H03922.}

\date{}
\begin{abstract}
We introduce and study a Hamiltonian formalism of mutations in cluster algebras
using canonical variables,
where the Hamiltonian is given by the Euler dilogarithm.
The corresponding Lagrangian, restricted to a certain subspace of the phase space,
coincides with the Rogers dilogarithm.
As an application,
we show how the dilogarithm identity  associated with a period of 
mutations in a cluster algebra
arises from the Hamiltonian/Lagrangian point of view.
%{\bf (ver.161105: not for destribution)}
\end{abstract}

\maketitle

\section{Introduction}

\subsection{Background and motivation}
\label{subsec:back1}

The {\em Euler dilogarithm}
\begin{align}
\mathrm{Li}_2(x)&=\sum_{n=1}^{\infty} \frac{x^n}{n^2}
=-
\int_0^{x}
\frac{\log (1-y)}{y}
dy
\end{align}
appears in several areas of mathematics
\cite{Zagier07}.
One of the main features of this function is
that it satisfies a wide variety of functional relations ({\em dilogarithm identities\/}),
many of which look miraculous and mysterious.
{\em Cluster algebras} are a class of commutative algebras
introduced by Fomin and Zelevinsky \cite{Fomin02}.
They originated in Lie theory
but turned out to be related to several areas of mathematics.

Fock and Goncharov
\cite{Fock07}  recognized that
the function $\mathrm{Li}_2(x)$ is  {\em built into} cluster algebra theory
{\em as a Hamiltonian}.
To see this, consider the Poisson bracket
introduced by \cite{Gekhtman02},
\begin{align}
\label{eq:bra1}
\{y_i,y_i\}=b_{ij}y_iy_j,
\end{align}
where $y_1$,\dots,$y_n$ are commutative variables
and $B=(b_{ij})_{i,j=1}^n$ is a skew-symmetric matrix.
Then, using formula \eqref{eq:pbf1}, we have
\begin{align}
\label{eq:ham1}
\{ \mathrm{Li}_2(-y_k),y_i\}
=
-b_{ki}\log(1+y_k) \cdot y_i.
\end{align}
This is an infinitesimal form (of the automorphism part) of the mutation
of $y$-variables ($X$-variables in \cite{Fock07}) in cluster algebras.
Therefore, one may regard the function $\mathrm{Li}_2(-y_k)$
as a Hamiltonian  with continuous time variable,
and the ordinary mutation is obtained as the time one flow of this Hamiltonian.
This viewpoint naturally guided the authors of  \cite{Fock03,Fock07b,Fock07} to quantize the cluster algebras using  the {\em quantum dilogarithm}.

Meanwhile, there is another story which developed independently, connecting the dilogarithm and cluster algebras,
where the {\em Rogers dilogarithm}
\begin{align}
%\label{eq:L1}
L(x)
=
-\frac{1}{2}
\int_0^{x}
\left\{
\frac{\log (1-y)}{y}
+
\frac{\log y}{1-y}
\right\}
dy
=\mathrm{Li_2}(x)+\frac{1}{2}\log x \log(1-x)
\end{align}
plays the central role.
The function $L(x)$ is a variant of $\mathrm{Li}_2(x)$
and it is known that  many dilogarithm identities are simplified
in terms of $L(x)$.
In the 90's, several conjectures on dilogarithm identities for $L(x)$
were given through the study of 
so called {\em $Y$-systems} in integrable models of  Yang-Baxter type.
Later, the connection between $Y$-systems and cluster algebras was recognized
\cite{Fomin03b},
and  these dilogarithm conjectures were solved using the cluster algebraic method
\cite{Chapoton05, Nakanishi09,Inoue10a,Inoue10b}
with the help of the {\em constancy condition} from \cite{Frenkel95}.
Then, these results were further generalized to the following theorem
\cite{Nakanishi10c}: {\em For any period in a cluster algebra,
there is an associated dilogarithm identity of $L(x)$}.

It is interesting to understand how these  seemingly independent
appearances of the Euler and Rogers dilogarithms
are intrinsically related.
In \cite{Kashaev11} it was clarified that 
the  dilogarithm identities in \cite{Nakanishi10c} are recovered
through the semiclassical analysis of the corresponding {\em quantum dilogarithm identities},
where the Rogers dilogarithm emerges through the relation
\begin{align}
\label{eq:Leg1}
L\left(\frac{x}{1+x}\right)
&=-\mathrm{Li_2}(-x)-\frac{1}{2}\log x \log(1+x).
\end{align}
Furthermore,
the result therein yields the following curious observation:
{\em
The relation \eqref{eq:Leg1} can be
regarded as the  Legendre transformation in classical mechanics,
where the Euler dilogarithm is the Hamiltonian, while the Rogers
dilogarithm is the Lagrangian.}
However,
to formulate and establish the claim precisely,  we need  {\em canonical variables} for the Poisson bracket
\eqref{eq:bra1}.

Motivated by the above question, in this paper we introduce
canonical variables of the Poisson bracket \eqref{eq:bra1},
and study the Hamiltonian and Lagrangian formalisms of
mutations in cluster algebras.
As a consequence, the above observation is justified;
furthermore,
we  can successfully explain how
the dilogarithm identities in \cite{Nakanishi10c}
naturally arise in the Hamiltonian/Lagrangian picture.

 \subsection{Outline and main results}
 Let us briefly describe the outline and the main results
of the paper.

In Section 2, we recall basic definitions and properties
of mutations in cluster algebras and of the dilogarithm functions
which we are going to use.

In Section 3, we introduce the Hamiltonian formalism of
mutations in cluster algebras with canonical variables.
The $x$- and $y$-variables in a cluster algebra
 are constructed as  exponentials of linear combinations of
 the canonical variables,
 while the Hamiltonian is given by the Euler dilogarithm.
The time one flow of the Hamiltonian,
together with the tropical transformation,
 yields
the mutation of the $x$- and $y$-variables
(Theorem \ref{thm:Hs3}). 
This naturally extends the Hamiltonian formalism in
\cite{Fock07}.
An interesting feature here is that the $y$-variables mutate properly on
the total phase space $M\simeq \mathbb{R}^{2n}$, while 
the $x$-variables do so only on
a certain subspace $M_0$ of the phase space.
This does not have an effect on the equations of motion;
but does affect the {\em quantization} of the Poisson bracket
in the following way:
\begin{itemize}
\item
The canonical quantization of the
Poisson bracket leads to the quantization of the $y$-variables by
\cite{Fock03,Fock07}.
\item
On the other hand, the Poisson bracket on the small phase space
$M_0$ is redefined via the {\em Dirac bracket} due to \cite{Dirac50}.
Then, 
the ``canonical quantization'' of the
Dirac bracket leads to the quantization of the $x$-variables by
\cite{Berenstein05b}.
\end{itemize}
Therefore, the formulation here provides a common platform
for the quantization of both $x$- and $y$-variables.

In Section 4, having the canonical variables at hand,
 we study the Lagrangian formalism of mutations.
A specific feature here is that
 the Hamiltonian in Section 3 is {\em singular}.
 This implies that
 we do not have a Lagrangian whose Euler-Lagrange
 equations are fully equivalent to the equations of motion of the Hamiltonian.
Despite this deficiency, one can still define the Lagrangian
 through the Legendre transformation.
 Then, the following fact holds:
 \par
% \noindent
 {\bf Fact 1.} (Proposition \ref{prop:Leg2})
 The Lagrangian coincides with the Rogers dilogarithm
 on the above small phase space $M_0$.
 \par
This justifies the observation stated in the previous subsection.

In Section 5, we make a little detour to establish
 the periodicity property of the canonical variables.
In the Hamiltonian formalism we naturally introduce a sign $\varepsilon=\pm$
to decompose  each seed mutation into the tropical and nontropical parts.
The mutation of  $x$- and $y$-variables are independent of the choice of the sign
$\varepsilon$, while
the mutation of the canonical variables depends on it.
Thus, we call these mutations  {\em signed mutations}.
We have the following result,
which is crucial for obtaining the final result in the next section.
 \par
%  \noindent
 {\bf Fact 2.} (Proposition \ref{prop:period2})
A sequence of  signed mutations enjoys the same periodicity
as a sequence of  seed mutations,
if we choose the sign sequence therein  as the {\em tropical sign sequence}.

In Section 6, we show how the dilogarithm identities  in \cite{Nakanishi10c}
arise from the Hamiltonian/Lagrangian point of view.
 This is done by considering the {\em action integral} of the singular
 Lagrangian from Section 4 along a Hamiltonian flow.
 There are two key facts:
  \par
 % \noindent
 {\bf Fact 3.} (Proposition \ref{prop:const1})
 The Lagrangian is piecewise constant along the flow.
  \par
 % \noindent
  {\bf Fact 4.}   (Theorem \ref{thm:const1})
 The  action integral does not depend on the flow,
 if all flows are periodic for the time span of the integral.
 (This is regarded as the converse of a finite time analogue of Noether's theorem,
 see Remark \ref{rem:noether1}.)
\par
 The dilogarithm identities   in \cite{Nakanishi10c} are
 obtained as  an immediate consequence
 of the above Facts 1--4.
 This is the main result of the paper
 (Theorems \ref{thm:const3} and \ref{thm:equiv1}).
 
\begin{rem}
Theorem 6.8  can be straightforwardly extended to
 generalized cluster algebras and dilogarithm functions of higher degree
 studied in \cite{Nakanishi16}.
Meanwhile,  Theorem 6.9 can be also extended to them
  if we admit  the version of Theorem 6.11 
 for generalized cluster algebras, which is not yet available.
\end{rem}
 
 %\bigskip
{\em Acknowledgements.} 
The idea of regarding the relation
between the Euler and Rogers dilogarithms appearing in the semiclassical analysis
in \cite{Kashaev11} as a Legendre transformation
was first suggested to T. N. by Junji Suzuki
at the workshop ``Infinite Analysis 11 Winter School: Quantum Cluster Algebras"
held at Osaka University in December, 2011.
We greatly thank him for his valuable suggestion.
This work is supported in part by JSPS Grant No. 16H03922 to T. N. and by M.G.'s NSF grant No. 1362801.

\section{Preliminaries}

\subsection{Mutations in cluster algebras}
\label{subsec:mut1}

Let us recall two main notions in cluster algebras,
namely, a {\em seed} and its {\em mutation}.
See  \cite{Fomin02,Fomin07} for more information
  on cluster algebras.

Let us fix a positive integer $n$ throughout the paper.
We say that an $n\times n$ integer matrix $B=(b_{ij})_{i,j=1}^n$
is
{\em skew-symmetrizable}
if there is  a diagonal matrix $D=\mathrm{diag}(d_1,\dots,d_n)$
with positive integer diagonal
entries $d_1,\dots,d_n$ such that $DB$ is skew-symmetric,
i.e., $d_ib_{ij}= - d_j b_{ji}$.
We call such $D$ a {\em skew-symmetrizer} of $B$.

Let us fix a semifield $\mathbb{P}$,
that is,  an abelian multiplicative group
with a binary operation $\oplus$ called the {\em addition},
which is commutative, associative, and distributive, i.e., $a(b\oplus c)=
ab \oplus ac$.
Let $\mathbb{Z}\mathbb{P}$ be the group ring of $\mathbb{P}$.
Since $\mathbb{Z}\mathbb{P}$ is a domain \cite{Fomin02},
 the field of fractions of $\mathbb{Z}\mathbb{P}$
 is well defined
and we denoted it by $\mathbb{Q}\mathbb{P}$.
Let $\mathcal{F}=\mathcal{F}_{\mathbb{P}}$ be a purely transcendental
field extension of $\mathbb{Q}\mathbb{P}$ of degree $n$,
that is $\mathcal{F}$ is isomorphic to
a rational function field of $n$ variables
with coefficients in $\mathbb{Q}\mathbb{P}$.
We call the semifield $\mathbb{P}$ and
the field $\mathcal{F}$ the {\em coefficient semifield}
and the {\em ambient field} (of a cluster algebra
under consideration), respectively.

A {\em seed with coefficients in $\mathbb{P}$}
is
 a triplet $(B,x,y)$ consisting of
an $n\times n$ skew-symmetrizable matrix $B$,
an $n$-tuple  $(x_i)_{i=1}^n$ of algebraically independent 
elements in $\mathcal{F}$,
and an  $n$-tuple  $(y_i)_{i=1}^n$ of elements in $\mathbb{P}$.
For each $k=1,\dots,n$, the {\em mutation of a seed $(B,x,y)$ at $k$}
is another seed $(B',x',y')=\mu_k(B,x,y)$,
which is obtained from $(B,x,y)$ by the following formulas:

\begin{align}
\label{eq:Bmut1}
b'_{ij}
&=
\begin{cases}
-b_{ij}
& \text{$i=k$ or $j=k$}\\
b_{ij}+[-\varepsilon b_{ik}]_+ b_{kj}
+ b_{ik} [\varepsilon b_{kj}]_+
&
i,j\neq k,
\end{cases}
\\
\label{eq:xmut1}
x_i '
&= 
\begin{cases}
\displaystyle
x_k^{-1} \bigg(\prod_{j=1}^n x_j^{[-
\varepsilon b_{jk}]_+}\bigg)
\frac{1+\hat{y}_k^{\varepsilon}}
{ 1\oplus y_k^{\varepsilon}} &i=k\\
x_i
& i\neq k,\\
\end{cases}
\\
\label{eq:ymut1}
y_i '&= 
\begin{cases}
y_k^{-1} &i=k\\
y_i y_k^{[\varepsilon b_{ki}]_+}(1\oplus y_k^{\varepsilon})^{-b_{ki}}
& i\neq k,\\
\end{cases}
\end{align}
where
\begin{align}
\label{eq:yhat2}
\hat{y}_i := y_i \prod_{j=1}^n x_j^{b_{ji}},
\end{align}
and $\varepsilon$ is a sign, $+$ or $-$, which is naturally identified with $1$ or $-1$,
respectively.
Then we have the following properties:
\par
(1). The right hand sides of \eqref{eq:Bmut1}--\eqref{eq:ymut1}
are  independent of
the choice of sign $\varepsilon$.
\par
(2).
If $D$ is a skew-symmetrizer of $B$,
then it is also a skew-symmetrizer of  $B'$.
\par
(3).
The mutation $\mu_k$ is involutive, namely,
\begin{align}
\label{eq:inv1}
\mu_k \circ \mu_k=\mathrm{id}.
\end{align}
\par
(4). The $\hat{y}$-variables \eqref{eq:yhat2} also mutate in $\mathcal{F}$ as the $y$-variables;
namely,
\begin{align}
\label{eq:yhmut1}
\hat{y}_i '&= 
\begin{cases}
\hat{y}_k^{-1} &i=k\\
\hat{y}_i \hat{y}_k^{[\varepsilon b_{ki}]_+}(1+ \hat{y}_k^{\varepsilon})^{-b_{ki}}
& i\neq k.\\
\end{cases}
\end{align}
%\begin{rem}
%We stress that the $y$-variables here are {\em not} the coefficients of the $x$-variables.
%Therefore, we can treat the $x$-and $y$-variables separately.
%\end{rem}

%For convenience we treat the $x$ and $y$-variables separately.

%Let us start from  the $x$-variables.
Let $ \mathcal{F}^{n}_0$ be the set of all $n$-tuples of
algebraically independent elements in $\mathcal{F}$,
and, as usual, let $\mathbb{P}^{n}$ be the set of all $n$-tuples of
elements in $\mathbb{P}$.
Let us extract the ``variable part"
of the mutation $\mu_k$
in  \eqref{eq:xmut1} and \eqref{eq:ymut1}
as
\begin{align}
\label{eq:mukb1}
\begin{matrix}
\mu^B_{k}: & \mathcal{F}^{n}_0\times \mathbb{P}^n &
\rightarrow 
&
 \mathcal{F}^{n}_0\times \mathbb{P}^n\\
&
(x,y) &
\mapsto
&
(x',y'),
\end{matrix}
\end{align}
and call it the {\em mutation at $k$ by $B$}.
The involution property \eqref{eq:inv1} is equivalent to
the  inversion relation,
\begin{align}
\label{eq:inv2}
\mu_{k}^{B_k}\circ \mu_k^B
= \mathrm{id},
\end{align}
where $B_k=B'$ is the one  in \eqref{eq:Bmut1}.

Following the idea of \cite{Fock03}, we decompose 
the mutation $\mu^B_{k}$
into two parts.
For each sign $\varepsilon =\pm$,
we introduce a map
\begin{align}
&\begin{matrix}
\rho^{B}_{k,\varepsilon}: &\mathcal{F}^{n}_0\times \mathbb{P}^n&
\rightarrow 
&
\mathcal{F}^{n}_0\times \mathbb{P}^n\\
&
(x,y) &
\mapsto
&
(\tilde{x},\tilde{y}),
\end{matrix}\\
\label{eq:auto2}
&
\displaystyle
\qquad\qquad
\tilde{x}_i = x_i 
\left(
\frac{1+\hat{y}_k^{\varepsilon}}
{1\oplus {y}_k^{\varepsilon}}
\right)
^{-\delta_{ki}},
\\
\label{eq:auto3}
&\qquad\qquad
\tilde{y}_i = y_i (1\oplus {y}_k^{\varepsilon})^{-b_{ki}},
\end{align}
and also a map
\begin{align}
\label{eq:tau7}
&\begin{matrix}
\tau^{B}_{k,\varepsilon}: &\mathcal{F}^{n}_0\times \mathbb{P}^n&
\rightarrow 
&
\mathcal{F}^{n}_0\times \mathbb{P}^n\\
&
({x},y) &
\mapsto
&
({x}',y'),
\end{matrix}\\
\label{eq:trop1}
&\qquad\qquad
x'_i =
\begin{cases}
\displaystyle
{x}_k^{-1}  \bigg(\prod_{j=1}^n {x}_j^{[-\varepsilon b_{jk} ]_+}\bigg)& i=k
\\
 {x}_i 
& i\neq k,
\end{cases}
\\
\label{eq:trop2}
&\qquad\qquad
y'_i =
\begin{cases}
\displaystyle
{y}_k^{-1} & i=k
\\
 {y}_i y_k^{[\varepsilon b_{ki}]_+} 
& i\neq k.
\end{cases}
\end{align}
Then, for each sign $\varepsilon=\pm$, the mutation $\mu^{B}_{k}$ is decomposed as
\begin{align}
\label{eq:decom1}
\mu^{B}_{k}=\tau^{B}_{k,\varepsilon}\circ \rho^{B}_{k,\varepsilon}.
\end{align}

 In \cite{Fock03},
the transformations
$\rho^{B}_{k,\varepsilon}$ and $\tau^{B}_{k,\varepsilon}$
for $\varepsilon=+$ were considered,
and they were called
the {\em automorphism part} and the {\em monomial part} of the mutation $\mu^B_{k}$,
respectively.
Here, we call 
$\rho^{B}_{k,\varepsilon}$ and $\tau^{B}_{k,\varepsilon}$
the {\em nontropical part} and the {\em tropical part} of the mutation $\mu^B_{k}$,
respectively.
See  \cite{Nakanishi11c} for the background of the terminology.

When the coefficient semifield $\mathbb{P}$ is taken to be the
{\em trivial semifield} $\mathbf{1}=\{1\}$, where $1\oplus 1=1$,
we say that the $x$-variables are {\em without coefficients\/}.
In that case
the transformations \eqref{eq:xmut1} and \eqref{eq:auto2} 
are simplified as
\begin{align}
\label{eq:xmut2}
x_i '
&= 
\begin{cases}
\displaystyle
x_k^{-1} \bigg(\prod_{j=1}^n x_j^{[-
\varepsilon b_{jk}]_+}\bigg)
(1+\hat{y}_k^{\varepsilon}) &i=k\\
x_i
& i\neq k,\\
\end{cases}
\\
\label{eq:auto4}
\displaystyle
\qquad\qquad
\tilde{x}_i &= x_i 
(1+\hat{y}_k^{\varepsilon})^{-\delta_{ki}},
\end{align}
respectively, while
\eqref{eq:yhat2} also reduces to
\begin{align}
\label{eq:yhatred1}
\hat{y}_i = \prod_{j=1}^n x_j^{b_{ji}}.
\end{align}

\subsection{Euler and Rogers dilogarithm functions}

Let us recall the definition of the Euler and Rogers dilogarithms.
See \cite{Lewin81,Zagier07} for more information.

The {\em Euler dilogarithm} $\mathrm{Li}_2(x)$ is originally defined as the following convergent series
with radius of convergence 1,
\begin{align}
\label{eq:Li1}
\mathrm{Li}_2(x)&=\sum_{n=1}^{\infty} \frac{x^n}{n^2}.
\end{align}
It has the integral expression
\begin{align}
\label{eq:Li2}
\mathrm{Li}_2(x)=-
\int_0^{x}
\frac{\log (1-y)}{y}
dy,
\quad
(x\leq 1),
\end{align}
where throughout the text we concentrate on the real region $x\leq 1$
so that there is no ambiguity due to multivaluedness of the integral.
Note that \eqref{eq:Li2} is also written as
\begin{align}
\label{eq:Li3}
\mathrm{Li}_2(-x)=
-
\int_0^{x}
\frac{\log (1+y)}{y}
dy,
\quad
(-1 \leq x).
\end{align}

On the other hand, the {\em Rogers dilogarithm} $L(x)$
is defined by the integral expression
\begin{align}
\label{eq:L1}
L(x)=-\frac{1}{2}
\int_0^{x}
\left\{
\frac{\log (1-y)}{y}
+
\frac{\log y}{1-y}
\right\}
dy,
\quad
(0\leq x\leq 1).
\end{align}
Again, since we concentrate on the real region $0\leq x\leq 1$,
there is no ambiguity due to multivaluedness of the integral.

These two dilogarithms are related by
\begin{align}
\label{eq:LLi1}
L(x)=\mathrm{Li_2}(x)+\frac{1}{2}\log x \log(1-x),
\quad
(0\leq x\leq 1),
\end{align}
which can be used as an alternative definition of the Rogers dilogarithm.
They are also related by the following less well-known formula:
\begin{align}
\label{eq:LLi2}
L\left(\frac{x}{1+x}\right)
&=-\mathrm{Li_2}(-x)-\frac{1}{2}\log x \log(1+x),
\quad
(0\leq x)\\
\label{eq:L3}
&=
\frac{1}{2}
\int_0^{x}
\left\{
\frac{\log (1+y)}{y}
-
\frac{\log y}{1+y}
\right\}
dy,
\quad
(0\leq x).
\end{align}
Formulas \eqref{eq:LLi1}--\eqref{eq:L3} can be most easily verified by
taking the derivative.

In view of the formulas \eqref{eq:LLi2} and \eqref{eq:L3},
it is convenient to introduce a function
\begin{align}
\label{eq:L4}
\tilde{L}(x)
&=
L\left(\frac{x}{1+x}\right)
=
\frac{1}{2}
\int_0^{x}
\left\{
\frac{\log (1+y)}{y}
-
\frac{\log y}{1+y}
\right\}
dy,
\quad
(0\leq x),
\end{align}
so that it satisfies the equality
\begin{align}
\label{eq:LLi3}
\tilde{L}(x)
&=-\mathrm{Li_2}(-x)-\frac{1}{2}\log x \log(1+x),
\quad
(0\leq x).
\end{align}
For simplicity, we still call the function $\tilde{L}(x)$ 
the Rogers dilogarithm.

%\begin{rem}
%The function $\tilde{L}(x)$  plays a key role in the {\em dilogarithm identities}
%associated with periods of cluster algebras \cite{Nakanishi10c}. 
%Namely, those dilogarithm identities take the form
%\begin{align}
%\sum_{t=1}^M \tilde{L}(y_{k_s}(t)) = \mathrm{const.},
%\end{align}
%where $y_{k_s}(t)$'s are certain components of $y$-variables in a cluster algebra.
%\end{rem}

\section{Hamiltonian formalism of mutations}

\subsection{Canonical and log-canonical variables}

Let $M$ be a symplectic manifold
with a global   Darboux chart $\varphi:M \buildrel \sim \over \rightarrow \mathbb{R}^{2n}$.
Let $(u,p)$,
 $u=(u_1,\dots, u_n)$,
 $p=(p_1,\dots, p_n)$,
be the {\em canonical coordinates}  of the chart.
Then, in the coordinates $(u,p)$, the {\em Poisson bracket} is given by
\begin{align}
\{f,g\} = \sum_{i=1}^n \left(\frac{\partial f}{\partial p_i}\frac{\partial g}{\partial u_i}
-
 \frac{\partial g}{\partial p_i}\frac{\partial f}{\partial u_i}
 \right)
\end{align}
for  any (smooth) functions  $f$ and $g$ on $M$.
We call $M$ the {\em phase space}. 
%We conveniently identify $M$ with $\mathbb{R}^{2n}$ by the
%global chart $\varphi$, unless otherwise mentioned.
%See  \cite{Abraham78}, for example,
%for the basics on the geometrical formulation of classical mechanics.
%

We recall some basic properties of the Poisson bracket which we use below.

\par
(1) We have
\begin{align}
\label{eq:pu1}
\{ p_i,u_j\}=\delta_{ij},
\quad
\{ u_i,u_j\}=\{ p_i,p_j\}=0.
\end{align}
\par
(2)
For any function $f$ on $M$,
\begin{align}
\{f,p_i\}=-\frac{\partial f}{\partial u_i},
\quad
\{f,u_i\}=\frac{\partial f}{\partial p_i}.
\end{align}
\par
(3) For any  functions $f$ and $g$ on $M$,
and any smooth functions $F(\zeta)$ and $G(\zeta)$ of a single variable $\zeta$,
\begin{align}
\label{eq:pbf1}
\{F(f),G(g)\}
=\{f,g\} F'(f) G'(g).
\end{align}
In particular,
the following formula holds:
\begin{align}
\label{eq:exp1}
\{e^f,e^g\}
=\{f,g\} e^f e^g.
\end{align}

Let us fix any $n\times n$ skew-symmetrizable (integer) matrix $B=(b_{ij})_{i,j=1}^n$
with a skew-symmetrizer $D$.
We introduce  variables (i.e., functions on $M$) $w_i, x_i, y_i$ ($i=1,\dots,n$)
as follows:
\begin{align}
w_i &= \sum_{j=1}^n b_{ji} u_j,\\
\label{eq:x1}
x_i & = e^{2 u_i},\\
\label{eq:y1}
y_i & = e^{d_ip_i+w_i}.
\end{align}

\begin{lem}
\label{lem:pw1}
We have the following formulas:
\begin{align}
\{d_ip_i+w_i,
d_j p_j+w_j\}
=2d_ib_{ij},
\quad
\{d_ip_i+w_i,
u_j\}
=d_i\delta_{ij}.
\end{align}
\end{lem}

\begin{defn}
Following \cite{Gekhtman02}, we say that a family of variables
$z_1$, \dots, $z_m$ is {\em log-canonical}
if their pairwise Poisson brackets are of the form
\begin{align}
\{z_i,z_j\} = c_{ij} z_i z_j,
\end{align}
where each $c_{ij}$ is a constant.
\end{defn}

\begin{prop}
The variables $x_1,\dots,x_n$ and $y_1,\dots,y_n$ are log-canonical with the following Poisson brackets:
\begin{align}
\label{eq:lc1}
\{x_i,x_j\}=0,
\quad
\{y_i,y_j\}=2d_i b_{ij}y_i y_j,
\quad
\{y_i,x_j\}=2d_i\delta_{ij}  y_i x_j.
\end{align}
\end{prop}
\begin{proof}
Follows from \eqref{eq:pu1}, \eqref{eq:exp1}, and Lemma \ref{lem:pw1}.
% For example,
%\begin{align}
%\begin{split}
%\{y_i,y_j\}&=
% \sum_{k=1}^n \left(\frac{\partial y_i}{\partial p_k}\frac{\partial y_j}{\partial u_k}
%-
% \frac{\partial y_j}{\partial p_k}\frac{\partial y_i}{\partial u_k}
% \right)
% \\
% &=
% \sum_{k=1}^n \left(\delta_{ik}d_iy_i b_{kj}y_j
%-
% \delta_{jk}d_j y_j b_{ki}y_i
% \right)\\
%& =y_i d_i b_{ij}y_j - y_j d_j b_{ji}y_i
% =2d_i b_{ij}y_i y_j.
% \end{split}
%\end{align}
\end{proof}

\subsection{Hamiltonian for infinitesimal nontropical mutation}

For any  $k\in \{1,\dots,n\}$
and a sign $\varepsilon=\pm$, we introduce the {\em Hamiltonian function}
$H^{B}_{k,\varepsilon}$ on $M$ by
\begin{align}
\label{eq:Hk1}
H^{B}_{k,\varepsilon}=\frac{\varepsilon}{2d_k} \mathrm{Li}_2(-y_k^{\varepsilon})
=-\frac{\varepsilon}{2d_k}\int_0^{y_k^{\varepsilon}} \frac{\log(1+z)}{z} dz,
\end{align}
where $ \mathrm{Li}_2(x)$ is the Euler dilogarithm  \eqref{eq:Li2},
and  we used the expression \eqref{eq:Li3}.

Let $t$ be the time variable.
We consider the Hamiltonian flow
on the phase space $M$
 by the Hamiltonian $H^{B}_{k,\varepsilon}$.
Accordingly,
 we have functions of $t$,
$u_i(t)$, $p_i(t)$, $w_i(t)$, etc.,
which
obey the following {\em equations of motion},
where we use the standard notation  $ \dot{f}=d f/dt$.
\begin{prop} 
\label{prop:Hs1}
(1)
The equations of motion are given as follows:
\begin{align}
\label{eq:Hu1}
\dot{u}_i(t)&=
\big\{H^{B}_{k,\varepsilon}, u_i(t)\big\}=
-\frac{1}{2}\delta_{ki} \log \big(1+y_k(t)^{\varepsilon}\big),
\\
\label{eq:Hp1}
 \dot{p}_i(t)&=
\big\{H^{B}_{k,\varepsilon}, p_i(t)\big\}=
-\frac{1}{2d_i}b_{ki} \log \big(1+y_k(t)^{\varepsilon}\big),
\\
\label{eq:Hw1}
\dot{w}_i(t)&=
\big\{H^{B}_{k,\varepsilon}, w_i(t)\big\}=
-\frac{1}{2}b_{ki} \log \big(1+y_k(t)^{\varepsilon}\big),
\\
\label{eq:Hdp1}
d_i \dot{p}_i(t)&=
\big\{H^{B}_{k,\varepsilon}, d_i p_i(t)\big\}=
-\frac{1}{2}b_{ki} \log \big(1+y_k(t)^{\varepsilon}\big),
\\
\label{eq:Hx1}
\dot{x}_i(t)&=
\big\{H^{B}_{k,\varepsilon}, x_i(t)\big\}=
-\delta_{ki} \log \big(1+y_k(t)^{\varepsilon}\big) \cdot x_i(t),
\\
\label{eq:Hy1}
\dot{y}_i(t)&=
\big\{H^{B}_{k,\varepsilon}, y_i(t)\big\}=
-b_{ki} \log \big(1+y_k(t)^{\varepsilon}\big) \cdot y_i(t).
\end{align}
\par
(2) In particular, $\dot{y}_k(t)=0$, so that
$y_k(t)$ in the right hand sides of
\eqref{eq:Hu1}--\eqref{eq:Hy1} does not depend on $t$.
\end{prop}
\begin{proof}
For example,
\begin{align}
\begin{split}
\big\{H^{B}_{k,\varepsilon}, u_i\big\}
&=\frac{\partial{H^{B}_{k,\varepsilon}}}{\partial p_i}
=\frac{d{H^{B}_{k,\varepsilon}}}{d y_k^{\varepsilon}}
\frac{d y_k^{\varepsilon}}{d y_k}
\frac{\partial{y_k}}{\partial p_i}
\\
&=
\left(-\frac{\varepsilon}{2 d_k} \frac{\log(1+y_k^{\varepsilon})}{y_k^{\varepsilon}}
\right)
(\varepsilon y_k^{\varepsilon-1})
(\delta_{ki}d_k y_k)
=
-\frac{1}{2}\delta_{ki} \log (1+y_k^{\varepsilon}),
\end{split}
\\
\begin{split}
\big\{H^{B}_{k,\varepsilon},  p_i\big\}
&=-\frac{\partial{H^{B}_{k,\varepsilon}}}{\partial u_i}
=-
\frac{d{H^{B}_{k,\varepsilon}}}{d y_k^{\varepsilon}}
\frac{d y_k^{\varepsilon}}{d y_k}
\frac{\partial{y_k}}{\partial u_i}
\\
&=
-
\left(-\frac{\varepsilon}{2 d_k} \frac{\log(1+y_k^{\varepsilon})}{y_k^{\varepsilon}}
\right)
(\varepsilon y_k^{\varepsilon-1})
(b_{ik}y_k)
=
-\frac{1}{2 d_i}b_{ki} \log (1+y_k^{\varepsilon}),
\end{split}
\end{align}
where $d_i b_{ik}= - d_k b_{ki}$ is used for the last equality.
\end{proof}

From Proposition \ref{prop:Hs1} one can also observe the following.
(In fact, it is a direct consequence of the fact that ${H^{B}_{k,\varepsilon}}$
only depends on the variable $y_k=\exp(d_k p_k + w_k)$.)
\begin{prop}
\label{prop:int1}
(1) The variables $u_1,\dots,u_{k-1},u_{k+1},\dots,u_n$ and $p_k$
are integrals of motion in involution.
Therefore, the Hamiltonian ${H^{B}_{k,\varepsilon}}$ is completely integrable.
\par
(2) The flows of the variables $u_k$ and $p_1,\dots,p_{k-1},p_{k+1},\dots,p_n$
are linear in $t$.
\end{prop}

Let us consider a {\em time one flow}
\begin{align}\label{eq:rho as flow}
\begin{matrix}
\rho^{B}_{k,\varepsilon}: & \mathbb{R}^{2n}
&\rightarrow &\mathbb{R}^{2n}\\
& (u,p) & \mapsto & (\tilde{u},\tilde{p}),
\end{matrix}
\end{align}
which is defined by
 the Hamiltonian flow
  from   time $t=0$ to $t=1$.
  Let $\tilde{w}_i, \tilde{x}_i,\tilde{y}_i$
  be the corresponding flows of $w_i, x_i, y_i$, respectively.
%As a general fact of Hamiltonian flows, the transformation $\rho^{B}_{k,\varepsilon}$ is canonical;
%namely, we have
%\begin{align}
%\{\tilde{p}_i,\tilde{u}_j\}=\delta_{ij},
%\quad
%\{\tilde{u}_i,\tilde{u}_j\}=0,
%\quad
%\{\tilde{p}_i,\tilde{p}_j\}=0.
%\end{align}
%\par
Let $\mathbb{R}_+$ be the semifield of all positive real numbers,
where the multiplication and the addition are given by the ordinary ones
for real numbers.

\begin{prop} 
\label{prop:Hs2}
We have the following formulas:
\begin{align}
\label{eq:Hu2}
\tilde{u}_i&=
u_i
-\frac{1}{2}
\delta_{ki} \log (1+y_k^{\varepsilon}),
\\
\label{eq:Hp2}
\tilde{p}_i&=
 p_i
-\frac{1}{2d_i}b_{ki} \log (1+y_k^{\varepsilon}),
\\
\label{eq:Hw2}
\tilde{w}_i&=
w_i
-\frac{1}{2}b_{ki} \log (1+y_k^{\varepsilon}),
\\
\label{eq:Hdp2}
d_i \tilde{p}_i&=
d_i p_i
-\frac{1}{2}b_{ki} \log (1+y_k^{\varepsilon}),
\\
\label{eq:Hx2}
\tilde{x}_i&=
x_i  (1+y_k^{\varepsilon})^{-\delta_{ki}},
\\
\label{eq:Hy2}
\tilde{y}_i&=
y_i  (1+y_k^{\varepsilon})^{-b_{ki} }.
\end{align}
In particular,
the transformation \eqref{eq:Hy2}
coincides with
the nontropical part of the mutation of
the $y$-variables in \eqref{eq:auto3}
with $\mathbb{P}=\mathbb{R}_+$.
%we idetify $\oplus$ in $\mathbb{P}$ with the addition in $\mathbb{R}_+$.
\end{prop}
\begin{proof}
This follows from Proposition
\ref{prop:Hs1}.
\end{proof}

Therefore,
the Hamiltonian $H^{B}_{k,\varepsilon}$ provides the  infinitesimal generator
of
 the nontropical
mutation of $y$-variables of seeds.
This Hamiltonian viewpoint of mutations (without employing the canonical variables
$u_i$ and $p_i$) was first stated in \cite[Section 1.3]{Fock07}
for $\varepsilon=1$.

\subsection{Small phase space and $x$-variables}

The transformation \eqref{eq:Hx2}
of $x$-variables is comparable
to the nontropical part of the mutation of
$x$-variables {\em without coefficients} from \eqref{eq:auto4}.
However, in  contrast to the $y$-variable case,
they do {\em not} exactly
match
due to the discrepancy between $y_k$ and $\hat{y}_k$ therein.
To remedy this situation,
we introduce  a subspace  of the phase space $M$,
\begin{align}
\label{eq:sp1}
M_0=\{ \varphi^{-1}(u,p)\in M
\mid d_i p_i - w_i=0,\ (i=1,\dots,n)
\},
\end{align}
and we call it the {\em small phase space}.

\begin{prop}
\label{prop:small2}
The small phase space $M_0$ is preserved under the Hamiltonian flow
by $H^{B}_{k,\varepsilon}$.
\end{prop}
\begin{proof}
This follows from 
 \eqref{eq:Hw1} and \eqref{eq:Hdp1}.
\end{proof}

Let us consider  the variables $\hat{y}_i$ ($i=1,\dots,n$)
in  \eqref{eq:yhat2} {\em without coefficients},
namely,
\begin{align}
\label{eq:yhat1}
\hat{y}_i &:= e^{2w_i} 
=e^{2\sum_{j=1}^n b_{ji} u_j}
=
 \prod_{j=1}^n x_j^{b_{ji}}.
\end{align}
%This is a particular case of
% \eqref{eq:yhat2} without coefficients.
Since $d_ip_i=w_i$ on $M_0$, we have
\begin{align}
\label{eq:yyat1}
y_i=\hat{y}_i
\quad
\text{on $M_0$.}
\end{align}
Thus,  the transformation \eqref{eq:Hx2} assumes the desired form on $M_0$:
\begin{prop}
\label{prop:small1}
\begin{align}
\label{eq:Hx3}
\tilde{x}_i&=
x_i  (1+\hat{y}_k^{\varepsilon})^{-\delta_{ki}}
\quad
\text{on $M_0$},
\end{align}
In particular,
the transformation \eqref{eq:Hx3}
restricted to $M_0$
coincides with
the nontropical part of the  mutation of
the $x$-variables without coefficients in \eqref{eq:auto4}
under the specialization of $x$-variables in $\mathbb{R}_+$.
\end{prop}

%\begin{rem}
%The map $M_0\rightarrow \mathbb{R}^n$ defined by $(u,p)\mapsto y$
%is surjective if and only if $B$ is invertible.
%\end{rem}

\subsection{Tropical transformation}
To complete the picture,
we also give a realization of the tropical transformations
\eqref{eq:trop1} and \eqref{eq:trop2} 
through a {\em change of coordinates} of the phase space $M$.
For any  $k\in \{1,\dots,n\}$
and a sign $\varepsilon=\pm$,
we consider the following  transformation:
\begin{align}
\label{eq:tau8}
\begin{matrix}
\tau^{B}_{k,\varepsilon}: &\mathbb{R}^{2n} &\rightarrow& \mathbb{R}^{2n}\\
&
({u},{p}) & \mapsto & (u',p')
\end{matrix}
\end{align}
\begin{align}
\label{eq:tau1}
u'_i
&=
\begin{cases}
\displaystyle
-{u}_k + \sum_{j=1}^n [-\varepsilon b_{jk}]_+ {u}_j
& i=k\\
{u}_i & i \neq k,
\end{cases}
\\
\label{eq:tau2}
p'_i
&=
\begin{cases}
-{p}_k 
& i=k\\
{p}_i +[-\varepsilon b_{ik}]_+ {p}_k & i \neq k.
\end{cases}
\end{align}

We call the transformation $\tau^{B}_{k,\varepsilon}$ a {\em tropical transformation}.
Note that it is an ordinary linear transformation,
not a {\em piecewise linear} transformation.

\begin{prop}
\label{prop:tau1}
Let $\tau^{B}_{k,\varepsilon}(u,p)=(u',p')$.
\par
(1) We have
\begin{align}
\label{eq:inner}
\sum_{i=1}^n u'_ip'_i
=
\sum_{i=1}^n u_ip_i.
\end{align}

\par
(2) The transformation $\tau^{B}_{k,\varepsilon}$ is canonical;
namely, we have
\begin{align}
\{p_i',u_j'\}=\delta_{ij},
\quad
\{u_i',u_j'\}=
\{p_i',p_j'\}=0.
\end{align}
\end{prop}
\begin{proof}
We write the linear transformations
\eqref{eq:tau1} and \eqref{eq:tau2} in the matrix form
as $u'=Mu$ and $p'=Np$.
Then, $N^T M=I$ holds.
Both properties (1) and (2) follow from it.
\end{proof}

By Proposition \ref{prop:tau1} (2), one can introduce a new
global  Darboux chart $\varphi':M\buildrel \sim \over\rightarrow \mathbb{R}^{2n}$
with canonical coordinates $(u',p')$
by the following commutative diagram:
\begin{align}
\xymatrix{
& M \ar[ld]_{\varphi} \ar[rd]^{\varphi'}\\
\mathbb{R}^{2n}\ar[rr]^{\tau^{B}_{k,\varepsilon}} &&\mathbb{R}^{2n}
}
\end{align}

Let $B'=B_k$.
We employ a common skew-symmetrizer $D$ for $B$ and $B'$,
and we define primed variables
$w'_i$, $x'_i$, $y'_i$
for $(u',p')=\tau^{B}_{k,\varepsilon}(u,p)$,
\begin{align}
w'_i &= \prod_{j=1}^n b'_{ji} u'_j,\\
x'_i&= e^{2u'_i},\\
y'_i & = e^{d_i p'_i + w'_i}.
\end{align}

\begin{prop} 
\label{prop:tau2}
We have the following formulas:
\begin{align}
\label{eq:tau3}
w'_i
&=
\begin{cases}
-w_k 
& i=k\\
w_i +[\varepsilon b_{ki}]_+ w_k & i \neq k,
\end{cases}
\\
\label{eq:tau4}
d_i p'_i
&=
\begin{cases}
-d_k p_k 
& i=k\\
d_i p_i +[\varepsilon b_{ki}]_+ d_k p_k & i \neq k,
\end{cases}
\\
\label{eq:tau6}
x_i '&=
\begin{cases}
\displaystyle
x_k^{-1}  \bigg(\prod_{j=1}^n x_j^{[-\varepsilon b_{jk} ]_+}\bigg)& i=k
\\
 x_i 
& i\neq k,
\end{cases}
\\
\label{eq:tau5}
y_i '&=
\begin{cases}
y_k^{-1} & i=k
\\
 y_i y_k^{[\varepsilon b_{ki}]_+}
& i\neq k.
\end{cases}
\end{align}
In particular,
the transformations \eqref{eq:tau6}  and
 \eqref{eq:tau5} 
coincide with
the tropical part of 
the  mutation
of the $x$- and $y$-variables
in
\eqref{eq:trop1}
and
\eqref{eq:trop2},
respectively.
\end{prop}
\begin{proof}
To prove \eqref{eq:tau3},
we use \eqref{eq:Bmut1} together with
\eqref{eq:tau1} and \eqref{eq:tau2}.
\end{proof}

\begin{prop}
In the new global Darboux chart $\varphi':M\buildrel \sim \over \rightarrow \mathbb{R}^{2n}$,
the small phase space $M_0$ is given by
\begin{align}
\label{eq:sp2}
M_0=\{ \varphi'{}^{-1}(u',p')\in M
\mid d_i p'_i - w'_i=0,\ (i=1,\dots,n)
\}.
\end{align}
\end{prop}
\begin{proof}
This follows from 
 \eqref{eq:tau3} and \eqref{eq:tau4}.
\end{proof}

\subsection{Signed mutations}

Let us introduce a composition of maps
\begin{align}
\label{eq:sm1}
\begin{matrix}
\mu^{B}_{k,\varepsilon} 
=
\tau^{B}_{k,\varepsilon} \circ \rho^{B}_{k,\varepsilon}: & \mathbb{R}^{2n}
&\rightarrow &\mathbb{R}^{2n}\\
& (u,p) & \mapsto & ({u}',{p}').
\end{matrix}
\end{align}

Summarizing
Propositions \ref{prop:Hs2},
\ref{prop:small1},
and \ref{prop:tau2},
 we have the following conclusion.
\begin{thm} 
\label{thm:Hs3}
Let $\mu^{B}_{k,\varepsilon} (u,p)=(u',p')$.
Then, we have the following formulas:
\begin{align}
\label{eq:Mu3}
{u}'_i&=
\begin{cases}
\displaystyle
-u_k + \sum_{j=1}^n [-\varepsilon b_{jk}]_+ u_j
+\frac{1}{2}
  \log(1+y_k^{\varepsilon})
& i=k\\
u_i & i \neq k,
\end{cases}
\\
\label{eq:Mp3}
 {p}'_i&=
\begin{cases}
- p_k 
& i=k\\
\displaystyle
p_i +[-\varepsilon b_{ik}]_+  p_k
-\frac{1}{2d_i}b_{ki} \log (1+y_k^{\varepsilon}) & i \neq k,
\end{cases}
\\
\label{eq:Mw3}
{w}'_i&=
\begin{cases}
-w_k 
& i=k\\
\displaystyle
w_i +[\varepsilon b_{ki}]_+ w_k
-\frac{1}{2}b_{ki} \log (1+y_k^{\varepsilon}),
 & i \neq k,
\end{cases}
\\
\label{eq:Mdp3}
d_i {p}'_i&=
\begin{cases}
-d_k p_k 
& i=k\\
\displaystyle
d_i p_i +[\varepsilon b_{ki}]_+ d_k p_k
-\frac{1}{2}b_{ki} \log (1+y_k^{\varepsilon}) & i \neq k,
\end{cases}
\\
\label{eq:Mx3}
{x}'_i&=
\begin{cases}
\displaystyle
x_k^{-1} \bigg(\prod_{j=1}^n x_j^{[-
\varepsilon b_{jk}]_+}\bigg)(1+{y}_k^{\varepsilon}) &i=k\\
x_i
& i\neq k,\\
\end{cases}
\\
\label{eq:My3}
{y}'_i&=
\begin{cases}
y_k^{-1} &i=k\\
y_i y_k^{[\varepsilon b_{ki}]_+}(1+y_k^{\varepsilon})^{-b_{ki}}
& i\neq k.\\
\end{cases}
\end{align}
In particular,
the transformation 
 \eqref{eq:My3} 
coincides with
the  mutation of the $y$-variables
in
\eqref{eq:ymut1}
with $\mathbb{P}=\mathbb{R}_+$,
while
the transformation 
 \eqref{eq:Mx3} restricted to $M_0$ 
coincides with
the  mutation of the $x$-variables without coefficients
in
\eqref{eq:xmut2}
under the specialization of $x$-variables in $\mathbb{R}_+$.
\end{thm}

As a corollary of
Theorem \ref{thm:Hs3},
$x'$ on $M_0$ and $y'$
therein
 do not depend on the choice of
the sign $\varepsilon$.
However,
$u'$ and $p'$
{\em do} depend on $\varepsilon$.
Thus, we call the map $\mu^{B}_{k,\varepsilon}$
in \eqref{eq:sm1}
the {\em signed mutation at $k$ by $B$ with sign $\varepsilon$}.

The inversion relation \eqref{eq:inv2} of the unsigned mutation
$\mu^{B}_{k}$ is replaced with the following one:
\begin{prop}
Define
$\rho^{B}_{k,\varepsilon}$
and $\rho^{B_k}_{k,-\varepsilon}$
by a common skew-symmetrizer $D$ of
$B$ and $B_k$.
Then, the following inversion relation holds:
\begin{align}
\label{eq:inv3}
\mu_{k,-\varepsilon}^{B_k}\circ \mu_{k,\varepsilon}^B
= \mathrm{id}.
\end{align}
\end{prop}
\begin{proof}
One can directly verify it by
\eqref{eq:Mu3} and \eqref{eq:Mp3}.
\end{proof}

\subsection{Canonical quantization}
\label{subsec:canonical}
One can canonically quantize the Poisson brackets
in \eqref{eq:pu1} 
by replacing them
 with the canonical commutation relations,
 \begin{align}
\label{eq:pu2}
[ P_i,U_j]=\frac{\hbar}{\sqrt{-1}}\delta_{ij},
\quad
[ U_i,U_j]=[ P_i,P_j]=0.
\end{align}
Then, we have
\begin{align}
\label{eq:dpw1}
[d_i P_i + W_i, d_j P_j + W_j]
= \frac{2\hbar}{\sqrt{-1}}d_i b_{ij}
=2\hbar \sqrt{-1} d_j b_{ji}.
\end{align}
Let us set
\begin{align}
q=e^{\hbar\sqrt{-1}}.
\end{align}
We recall a special case of the Baker-Campbell-Hausdorff formula.
For any noncommutative variables $A$ and $B$  such that
$[A,B]=C$ and $[C,A]=[C,B]=0$, we have
\begin{align}
e^A e^B = e^{C/2} e^{A+B},
\end{align}
or
\begin{align}
e^A e^B = e^{C} e^{B}e^A.
\end{align}
Applying it for \eqref{eq:dpw1},
we have the commutation relation
for $Y_i=e^{d_i P_i + W_i}$,
\begin{align}
Y_i Y_j = q^{2d_j b_{ji}} Y_j Y_i.
\end{align}
This coincides with the quantization of $y$-variables due to
Fock and Goncharov \cite{Fock03,Fock07}.

\begin{rem}
The realization of quantum $y$-variables by the canonical variables
 presented here appeared
in \cite{Fock07b,Kashaev11,Nakanishi16}.
In fact, the construction of $x$- and $y$-variables in 
\eqref{eq:x1} and \eqref{eq:y1} is deduced from the  quantum ones in \cite{Kashaev11,Nakanishi16}.
\end{rem}

\subsection{Quantization of $x$-variables through Dirac bracket}
\label{subsec:quantization}
Since $x$-variables are in involution for
the Poisson bracket \eqref{eq:lc1},
the canonical quantization in \eqref{eq:pu2}
 only provides the trivial quantization for them.
However, 
by Proposition \ref{prop:small1},
they should be restricted  to the small phase space $M_0$
to be identified with the $x$-variables in a seed.
Therefore, we should apply Dirac's method \cite{Dirac50}
 to obtain the Poisson structure on $M_0$.

Recall that the space $M_0$ is given by a family of constraints
 $\chi_i=0$ ($i=1,\dots,n$),
 where
\begin{align}
\chi_i=d_i p_i - w_i.
\end{align}
Following \cite{Dirac50},
let us consider the $n\times n$ matrix $A=(a_{ij})_{i,j=1}^n$
defined by
\begin{align}
a_{ij}:=\{\chi_i,\chi_j\}=-2d_i b_{ij}.
\end{align}
To proceed, we have to assume that
the matrix $A=-2DB$ is invertible,
or equivalently,
$B$ is invertible.
Then, we have $A^{-1}=-(1/2)B^{-1} D^{-1}$.

%\bigskip
%\par
%{\em (a). The case when $B$ is invertible.}

\begin{lem} The matrix $B^{-1}$ is skew-symmetrizable with $D$ being its
skew-symmetrizer.
\end{lem}
\begin{proof}
By assumption, we have $DB D^{-1}=-B^T$.
Taking its inverse, we have
 $DB^{-1} D^{-1}=-(B^{-1})^T$.
\end{proof}

The {\em Dirac bracket} is defined by
\begin{align}
\label{eq:Dbra1}
\{f,g\}_D
:=\{f,g\} - \sum_{i,j=1}^n \{f,\chi_i\} (A^{-1})_{ij}\{\chi_j, g\},
\end{align}
where  $(A^{-1})_{ij}$ is the $(i,j)$-component of $A^{-1}$.

Here are some basic properties of the Dirac bracket.

\par
(1) It defines a new Poisson bracket on $M$.
\par
(2) For any constraint function $\chi_i$ and any function $f$ on $M$,
\begin{align}
\{f,\chi_i \}_D = 0
\end{align}
holds.
Thus, for  any function $g$ on $M$,
$\{f,g\chi_i \}_D=\{f,g \}_D\chi_i$ vanishes on $M_0$.
As a consequence, it defines a Poisson bracket on $M_0$.
\par
(3)
For any function $f$ on $M$,
\begin{align}
\{H^{B}_{k,\varepsilon}, f\}_D
=
\{H^{B}_{k,\varepsilon}, f\},
\end{align}
since $\{H^{B}_{k,\varepsilon}, \chi_i\}=0$
as stated in Proposition \ref{prop:small2}.
Therefore, the equations of  motion do not change.

It is convenient to set
 $B^{-1}=\Omega=(\omega_{ij})_{i,j=1}^n$.
\begin{prop} 
We have the following formulas:
\begin{align}
\label{eq:pu3}
\{ p_i,u_j\}_D=\frac{1}{2}\delta_{ij},
\quad
&\{ u_i,u_j\}_D=-\frac{1}{2} d_i \omega_{ij},
\quad
\{ p_i,p_j\}_D=\frac{1}{2d_j}b_{ij},
\\
\label{eq:lc3}
\{x_i,x_j\}_D=-2d_i\omega_{ij}x_ix_j,
\quad
&\{\hat{y}_i,\hat{y}_j\}_D=2d_i b_{ij}\hat{y}_i \hat{y}_j,
\quad
\{\hat{y}_i,x_j\}_D=2 d_i \delta_{ij} \hat{y}_i x_j,
\end{align}
where $\hat{y}_i$ is defined by \eqref{eq:yhat1}.
\end{prop}
\begin{proof}
The formulas in \eqref{eq:pu3} are obtained by  explicit calculations
from the definition \eqref{eq:Dbra1}.
Note that $\hat{y}_i=e^{2d_i p_i}$ on $M_0$.
Then, we apply \eqref{eq:exp1} to obtain \eqref{eq:lc3}.
\end{proof}

In particular, the Dirac brackets in \eqref{eq:lc3} for $x$- and $\hat{y}$-variables
 coincide with the
Poisson brackets in \cite{Gekhtman02}.

Now we ``canonically quantize" the Dirac brackets
in \eqref{eq:pu3}
 by replacing them
 with the  commutation relations,
 \begin{align}
\label{eq:pu4}
[ P_i,U_j]=\frac{\hbar}{\sqrt{-1}}\frac{1}{2}\delta_{ij},
\quad
[ U_i,U_j]=
\frac{\hbar}{\sqrt{-1}}\frac{-1}{2}d_i \omega_{ij},
\quad
[ P_i,P_j]=
\frac{\hbar}{\sqrt{-1}}\frac{1}{2d_j}b_{ij}.
\end{align}

Then, in the same manner as in the previous subsection,
 we obtain 
the following commutation relation  for $X_i=e^{2U_i}$,
\begin{align}
X_i X_j = q^{2d_i \omega_{ij}} X_j X_i.
\end{align}
This coincides with the quantization of $x$-variables (without coefficients) due to
Berenstein and Zelevinsky \cite{Berenstein05b},
where the skew-symmetric matrix $\Lambda$ therein is related to
$\Omega$ via
$\Lambda^T=D\Omega$, and $q$ therein is identified with $q^{-2}$ here.

\section{Lagrangian  formalism and Rogers dilogarithm}

\subsection{Legendre transformation}
Let us recall some basic facts on the Legendre transformation of
a  Hamiltonian.
See, for example, \cite{Abraham78,Takhtajan08}
for more information.

For simplicity, let us consider a Hamiltonian $H$
on the space $\mathbb{R}^{2n}$ with the canonical coordinates $(u,p)$.
The  space $\mathbb{R}^{2n}$ is  naturally identified with the cotangent bundle 
$\pi: T^*\mathbb{R}^n\rightarrow \mathbb{R}^n $,
where $\pi(u,p)=u$ and $p=(p_i)_{i=1}^n$ represents the
1-form $\sum_{i=1}^n p_i du_i$.
Then, the Hamiltonian $H$ induces the following
fiber preserving map:
\begin{align}
\label{eq:reg1}
\begin{matrix}
F_H:&T^*\mathbb{R}^n&  \rightarrow & T\mathbb{R}^n\\
& (u,p) & \mapsto & (u,\dot{u}).
\end{matrix}
\end{align}
\begin{defn}
We say that a Hamiltonian $H(u,p)$ is {\em regular} if the map $F_H$ is a diffeomorphism.
\end{defn}

The {\em Lagrangian} $\mathscr{L}$ for the Hamiltonian $H$ is formally defined 
by the  {\em Legendre transformation},
\begin{align}
\label{eq:lag1}
\mathscr{L}= \sum_{i=1}^n \dot{u}_i p_i - H.
\end{align}
Here we use the symbol $\mathscr{L}$ so that it is not confused
with the Rogers dilogarithm $L(x)$ or $\tilde{L}(x)$.
From the definition $\mathscr{L}$ is a function of $(u,p)\in T^*\mathbb{R}^n$.
Assume that the Hamiltonian $H$ is regular.
Then, by the inverse map of $F_H$,
one can convert it to a function of $(u,\dot{u})$,
which is the Lagrangian function $\mathscr{L}(u,\dot{u})$.
The equations of motion of the Hamiltonian $H$ are equivalent to
the {\em Euler-Lagrange equations}
\begin{align}
\label{eq:leq2}
\frac{d}{dt}
\left(
\frac{\partial \mathscr{L}}{\partial \dot{u}_i}
\right)=  \frac{\partial \mathscr{L}}{\partial {u}_i},
\end{align}
together with the identification of variables $p_i$,
\begin{align}
\label{eq:leq1}
p_i = \frac{\partial \mathscr{L}}{\partial \dot{u}_i}.
\end{align}

There is a parallel notion of regularity for a Lagrangian.

\begin{defn}
We say that
a Lagrangian $\mathscr{L}(u,\dot{u})$ is  {\em regular}
if the map
\begin{align}
\label{eq:reg2}
\begin{matrix}
F_{\mathscr{L}}:&T\mathbb{R}^n&  \rightarrow & T^*\mathbb{R}^n\\
& (u,\dot{u}) & \mapsto & (u,p).
\end{matrix}
\end{align}
is a diffeomorphism,
where $p_i$ is defined by \eqref{eq:leq1}.
\end{defn}
It is known (e.g., \cite[Sec. 3.6]{Abraham78}) that from a regular Hamiltonian
one obtains a regular Lagrangian by the Legendre transformation;
conversely, 
from a regular Lagrangian
one obtains a regular Hamiltonian by the Legendre transformation.
In either case the two systems are equivalent in the above sense.

\subsection{Lagrangian and Rogers dilogarithm}
\label{subsec:lag1}
Let us consider the Hamiltonian 
$H=H^{B}_{k,\varepsilon}$ from \eqref{eq:Hk1} in the canonical coordinates $(u,p)$.
By \eqref{eq:Hu1},
we have
\begin{align}
\label{eq:dotu1}
\dot{u}_i&=
-\frac{1}{2}\delta_{ki} \log (1+y_k^{\varepsilon}).
\end{align}
Thus, the  map $F_H$ in \eqref{eq:reg1} is far from surjective.
Therefore, the Hamiltonian $H$ is singular (i.e., not regular), unfortunately.

Nevertheless, let us write the Lagrangian in \eqref{eq:lag1} explicitly,
for now, as a function of $(u,p)$,
\begin{align}
\label{eq:lag2}
\mathscr{L}^{B}_{k,\varepsilon}(u,p)
&=
-\frac{1}{2}\log (1+y_k^{\varepsilon})
p_k -
\frac{\varepsilon}{2d_k} \mathrm{Li}_2(-y_k^{\varepsilon}).
\end{align}
Inverting the relation \eqref{eq:dotu1} for $i=k$, we regard $y_k$ as a function of $\dot{u}_k$.
Then, we also regard $p_k$
as a function of $\dot{u}_k$ and $u_1,\dots,
u_{k-1},u_{k+1},\dots,u_n$ by the relation
\begin{align}
\label{eq:pk1}
p_k=d_k^{-1}(\log y_k- w_k).
\end{align}
Thus, the function  $\mathscr{L}^{B}_{k,\varepsilon}(u,p)$ is
converted to a function of 
$(u,\dot{u})$ by
\begin{align}
\label{eq:lag3}
\mathscr{L}^{B}_{k,\varepsilon}
(u,\dot{u})
&=
-\frac{1}{2d_k}\log (1+y_k^{\varepsilon})
(\log y_k - w_k) -
\frac{\varepsilon}{2d_k} \mathrm{Li}_2(-y_k^{\varepsilon}),
\end{align}
 despite the fact that the Hamiltonian is singular.

Of course, we have to pay some price.
The Lagrangian $\mathscr{L}^{B}_{k,\varepsilon}(u,\dot{u})$ is {\em singular},
since
it is independent of the variables $\dot{u}_i$ for $i \neq k$.
Moreover, it is {\em not} equivalent to the original Hamiltonian system.
\begin{prop}
\label{prop:le1}
(1) For $i=k$, the equation
\eqref{eq:leq2}, together with \eqref{eq:leq1},
yields
\begin{align}
\label{eq:Hu4}
y_k(t)=e^{d_k p_k(t) + w_k(t)},
\quad
 \dot{p}_k(t)=0.
\end{align}
(2) For $i\neq k$, the equation
\eqref{eq:leq2}, together with \eqref{eq:leq1},
yields
\begin{align}
\label{eq:Hu5}
p_i(t)=0,
\quad
b_{ki}\log (1+y_k(t)^{\varepsilon})=0.
\end{align}
\end{prop}
\begin{proof}
(1) For $i=k$,
\begin{align}
\label{eq:le1}
\frac{\partial \mathscr{L}}{\partial \dot{u}_k}
&=
\frac{1}{d_k}(\log y_k - w_k),
\\
\label{eq:le2}
\frac{\partial \mathscr{L}}{\partial {u}_k}
&=0.
\end{align}
(2) 
For $i\neq k$,
\begin{align}
\label{eq:le3}
\frac{\partial \mathscr{L}}{\partial \dot{u}_i}
&=
0,
\\
\label{eq:le4}
\frac{\partial \mathscr{L}}{\partial u_i}
&= 
\frac{b_{ik}}{2d_k}\log (1+y_k^{\varepsilon})
=
- \frac{b_{ki}}{2d_i}\log (1+y_k^{\varepsilon}).
\end{align}
\end{proof}

Therefore, the Euler-Lagrange equation for
$i=k$ is a part of the equations of motion
of the original Hamiltonian system in Proposition \ref{prop:Hs1}.
On the other hand, the ones for $i\neq k$ set an unwanted restriction
\eqref{eq:Hu5},
and  we have to avoid using them.
%which also forces to take the limit 
%\begin{align}
%p_k(t)\rightarrow \mp \infty \quad\text{for $\varepsilon=\pm$}
%\end{align}
%to have $y_k^{\varepsilon}=0$.

Putting this defect aside, let us evaluate  $\mathscr{L}^{B}_{k,\varepsilon}$
on the small phase space $M_0$.
Recall that we have $y_k = e^{2d_kp_k}$ on $M_0$.
Thus, we have
\begin{align}
p_k = \frac{1}{2d_k}\log y_k
=\frac{\varepsilon}{2d_k}\log y_k^{\varepsilon}
\quad
\text{on $M_0$}.
\end{align}
Thus,  $\mathscr{L}^{B}_{k,\varepsilon}$  depends
only on $y_k$, or equivalently, on $y_k^{\varepsilon}$.
So, let us write it is as a function of $y_k^{\varepsilon}$ as
 $\mathscr{L}^{B}_{k,\varepsilon}(y_k^{\varepsilon})$
for our convenience.
Then, putting it in \eqref{eq:lag2},
we obtain
\begin{align}
\label{eq:lm1}
\mathscr{L}^{B}_{k,\varepsilon}(y_k^{\varepsilon})
=
-\frac{\varepsilon}{4d_k}\log y_k^{\varepsilon}
\log (1+y_k^{\varepsilon})
-\frac{\varepsilon}{2d_k} \mathrm{Li}_2(-y_k^{\varepsilon}).
\end{align}

Now the Rogers dilogarithm emerges in our picture.
\begin{prop}
\label{prop:Leg2}
The function $\mathscr{L}^{B}_{k,\varepsilon}(y_k^{\varepsilon})$
 is given by the  Rogers dilogarithm $\tilde{L}(x)$ in \eqref{eq:L4} as
\begin{align}
\label{eq:LL1}
\mathscr{L}^{B}_{k,\varepsilon}(y_k^{\varepsilon})
= \frac{\varepsilon}{2d_k} \tilde{L}(y_k^{\varepsilon}).
\end{align}
\end{prop}
\begin{proof}
This follows from \eqref{eq:LLi3} and \eqref{eq:lm1}.
\end{proof}

Let us  boldly phrase the above result   as
{\em ``the Rogers dilogarithm is a Legendre transformation of 
the Euler dilogarithm.''}
This justifies the observation stated in Section \ref{subsec:back1}.

%\subsection{Perturbation of Hamiltonian}
%{\bf (Probably we do not use it, but temporarily I keep it. )} 
%Here we introduce a perturbation of the singular Hamiltonian 
%$H^{B}_{k,\varepsilon}$ to make it regular,
%so that the Lagrangian formalism is fully applicable.
%This will be used in the application to prove dilogarithm identities later.
%
%Let $H=H^{B}_{k,\varepsilon}$ be the original Hamiltonian.
%For a small  real number $\lambda>0$,
%we introduce a perturbation
%of $H$,
%\begin{align}
%H_{\lambda}=H -\lambda \frac{\varepsilon}{2} \sum_{i=1}^n p_i^2.
%\end{align}
%The equation of motion for $u_i$ is
%\begin{align}
%\label{eq:Hu3}
%\dot{u}_i(t)&=
%\{H_{\lambda}, u_i(t)\}=
%-\frac{1}{2}\delta_{ki} \log (1+y_k(t)^{\varepsilon})
%-\varepsilon {\lambda} p_i.
%\end{align}
%Then, the Jacobian
%\begin{align}
%\det_{1\leq i,j\leq n}
%\left(
% \frac{\partial \dot{u}_i }{\partial p_j}
%\right)
%=(-\varepsilon \lambda)^{n-1}
%\left(
%-\frac{\varepsilon}{2}
%\frac{d_ky_k^{\varepsilon -1}}{1+{y_k^{\varepsilon}}}
%-\varepsilon \lambda
%\right)
%\end{align}
%does not vanish.
%Thus, the Hamiltonian $H_{\lambda}$ is regular.
%Therefore, the Lagrangian $\mathscr{L}_{\lambda}$ of $H_{\lambda}$
%reproduces the Hamiltonian flow of $H_{\lambda}$,
%and by taking the limit $\lambda\rightarrow 0$,
%one can recover the
%Hamiltonian flow of $H$.
%
\section{Periodicity of canonical variables}

\subsection{Universal and tropical semifields}
Let us recall two important classes of semifields, following \cite{Fomin07}.

\begin{defn}
Let $y=(y_1,\dots,y_n)$ be an $n$-tuple of formal commutative variables. 
\par
(1)
Define a semifield 
\begin{align}
\begin{split}
\mathbb{Q}_+(y)
=\biggl\{
\frac{p(y)}{q(y)}
\in \mathbb{Q}(y)
\mid
&\text{
$p(y)$ and $q(y)$ are nonzero polynomials of $y$}\\
&\hskip35pt\text{with nonnegative integer coefficients}
\biggr\},
\end{split}
\end{align}
where
the multiplication and the addition $\oplus$ are given by
the ordinary ones for the rational function field $\mathbb{Q}(y)$.
We call it the {\em universal semifield of $y$}.
\par
(2)
Define a semifield 
\begin{align}
\begin{split}
\mathrm{Trop}(y)
=\left\{
\prod_{i=1}^n
y_i^{a_i}
\mid
a_i\in \mathbb{Z}
\right\},
\end{split}
\end{align}
where
the multiplication is given by the ordinary one 
for monomials of $y$,
while
 the addition $\oplus$ is given by
 the {\em tropical sum}
\begin{align}
\left(
\prod_{i=1}^n
y_i^{a_i}
\right)
\oplus
\left(
\prod_{i=1}^n
y_i^{b_i}
\right)
=
\prod_{i=1}^n
y_i^{\min(a_i,b_i)}.
\end{align}
We call it the {\em tropical semifield of $y$}.
\end{defn}

There is a semifield homomorphism
\begin{align}
\label{eq:tropmap1}
\begin{matrix}
\pi_{\mathrm{trop}}&:
\mathbb{Q}_+(y)
&\rightarrow
&\mathrm{Trop}(y)\\
& y_i 
&\mapsto
& y_i,
\end{matrix}
\end{align}
which we call the {\em tropicalization map}.

\subsection{Periodicity of seeds and tropical periodicity}

Let $y=(y_1,\dots,y_n)$ be an $n$-tuple of formal commutative variables.
We consider a sequence of seed mutations  {\em with
coefficients in the universal semifield $\mathbb{Q}_+(y)$},
\begin{align}
\label{eq:seq1}
\begin{split}
(B,x,y)=(B[0],x[0],y[0])
\buildrel \mu_{k_0} \over \mapsto
(B[1],x[1],y[1])
\buildrel \mu_{k_1} \over \mapsto
\cdots \\
\cdots
\buildrel \mu_{k_{T-1}} \over \mapsto
(B[T],x[T],y[T]).
\end{split}
\end{align}
Note that the initial $y$-variables $y$ are set to be the generators $y$ of 
$\mathbb{Q}_+(y)$.

We may view the sequence \eqref{eq:seq1} as a discrete dynamical system
with  a discrete time $s=0,1,\dots, T$.
Let us introduce the notion of periodicity for this system.

Let $\mathcal{F}_0^{n}$ be the one defined in Section
\ref{subsec:mut1} with $\mathbb{P}=\mathbb{Q}_+(y)$.
We define a (left) action of a permutation $\sigma$ of $\{1,\dots,n\}$ on
$\mathcal{F}^{n}_0\times \mathbb{Q}_+(y)^n$ by
\begin{align}
\label{eq:sigma1}
&\begin{matrix}
\sigma: &\mathcal{F}^{n}_0\times  \mathbb{Q}_+(y)^n&
\rightarrow 
&
\mathcal{F}^{n}_0\times  \mathbb{Q}_+(y)^n\\
&
({x},z) &
\mapsto
&
({x}',z'),
\end{matrix}\\
\label{eq:sigma2}
&\qquad\qquad
x'_i =
x_{\sigma^{-1}(i)},
\\
\label{eq:sigma3}
&\qquad\qquad
z'_i =
z_{\sigma^{-1}(i)}.
\end{align}

\begin{defn}
\label{defn:period1}
Let $\sigma$ be a permutation of $\{1,\dots,n\}$.
We say that a sequence of mutations
\eqref{eq:seq1} is {\em $\sigma$-periodic}
if the following conditions hold for any $1\leq i,j\leq n$:
\begin{align}
\label{eq:bperiod1}
b_{\sigma^{-1}(i)\sigma^{-1}(j)}[T]=b_{ij}[0],\\
\label{eq:xperiod1}
x_{\sigma^{-1}(i)}[T]=x_{i}[0],\\
\label{eq:yperiod1}
y_{\sigma^{-1}(i)}[T]=y_{i}[0].
\end{align}

\end{defn}

Note that the conditions 
\eqref{eq:xperiod1} and \eqref{eq:yperiod1} are also expressed as
the following equality  on $\mathcal{F}^n_0
\times  \mathbb{Q}_+(y)^n$:
\begin{align}
\label{eq:period4}
\sigma
\circ\mu^{B[T-1]}_{k_{T-1}}
\cdots
\circ\mu^{B[1]}_{k_1}
\circ\mu^{B[0]}_{k_0}
=
\mathrm{id}.
\end{align}
%Indeed, the left hand side of \eqref{eq:period4} sends $(x,y)$ to $(x,y)$
%for the generatorts $y=(y_1,\dots,y_n)$ of $\mathbb{Q}_+(y)$
%by \eqref{eq:xperiod1} and \eqref{eq:yperiod1}.
%For a  general $z\in \mathbb{Q}_+(y)^n$, let 
% $\varphi: \mathbb{Q}_+(y)
%\rightarrow \mathbb{Q}_+(y)$ be the homomorphism such hat
%$\varphi( y_i )=z_i$. Then, $\varphi$ commutes with
%the mutations and $\sigma$.
%Therefore, 
%the left hand side of \eqref{eq:period4} sends $(x,z)$ to $(x,z)$.

The following fact is known.
\begin{prop}
[{\cite[Proposition 4.3]{Nakanishi16}}]
\label{prop:d1}
Let $D=\mathrm{diag}(d_1,\dots,d_n)$ be any common skew-symmetrizer of $B[s]$ ($s=0,\dots,T-1$).
Suppose that the condition
\eqref{eq:bperiod1} holds.
Then, the following equality holds:
\begin{align}
\label{eq:d1}
d_{\sigma(i)}=d_i.
\end{align}
\end{prop}

Let us consider the ``tropicalization" of the sequence
\eqref{eq:seq1}.
By applying the tropicalization map $\pi_{\mathrm{trop}}$
in \eqref{eq:tropmap1} to each $y$-variable $y_i[s]$,
($s=0,\dots,T$) in the sequence
\eqref{eq:seq1},
we obtain a monomial of initial $y$-variables $y$,
\begin{align}
\label{eq:cmat1}
\pi_{\mathrm{trop}}(y_i[s])=
\prod_{j=1}^n
y_j^{c_{ji}[s]}.
\end{align}
The integer vector $c_i[s]=(c_{ji}[s])_{j=1}^n$ is called the
{\em $c$-vector} of $y_i[s]$.

The following fact is of fundamental importance
in the theory of cluster algebras.

\begin{thm}
[{(Sign-coherence of $c$-vectors),
\cite[Theorem 1.7]{Derksen10} with \cite[Proposition 5.6]{Fomin07},
\cite[Corollary 5.5]{Gross14}}]
\label{thm:sign1}
Each c-vector is a nonzero  vector, and
its components are either all nonnegative or all nonpositive.
\end{thm}

Based on this theorem, we define the following notion.
\begin{defn}
The {\em tropical sign $\varepsilon=\varepsilon(y_i[s])$ of $y_i[s]$} is given 
by  $+$ (resp. $-$) if  the components of the $c$-vector of $y_i[s]$ are
all nonnegative (resp. nonpositive).
\end{defn}

We introduce a sequence of signs,
$\varepsilon_0,\dots,\varepsilon_{T-1}$,
where
\begin{align}
\label{eq:tropss1}
\varepsilon_s =\varepsilon(y_{k_s}[s]),
\quad
(s=0,\dots,T-1),
\end{align}
and $k_s$ is the one in \eqref{eq:seq1}.
We call it the {\em tropical sign sequence} of \eqref{eq:seq1}.
Accordingly, we have the following sequence of
  transformations associated with
the sequence \eqref{eq:seq1}:

\begin{align}
\label{eq:tropseq1}
\begin{split}
\mathcal{F}^n_0
\times  \mathbb{Q}_+(y)^n
\buildrel \tau^{B[0]}_{k_0,\varepsilon_0} \over \rightarrow
\mathcal{F}^n_0
\times  \mathbb{Q}_+(y)^n
\buildrel \tau^{B[1]}_{k_{1},\varepsilon_{1}}
\over\rightarrow
\cdots
\buildrel \tau^{B[T-1]}_{k_{T-1},\varepsilon_{T-1}}\over \rightarrow
\mathcal{F}^n_0
\times  \mathbb{Q}_+(y)^n,
\end{split}
\end{align}
where $\tau^B_{k,\varepsilon}$ is the one in \eqref{eq:tau7}.

\begin{defn}
\label{defn:period2}
Let $\sigma$ be a permutation of $\{1,\dots,n\}$.
We say that a sequence of  transformations
\eqref{eq:tropseq1} is {\em $\sigma$-periodic}
if the following equality holds
 on $\mathcal{F}^n_0
\times  \mathbb{Q}_+(y)^n$:
\begin{align}
\label{eq:period1}
\sigma
\circ\tau^{B[T-1]}_{k_{T-1},\varepsilon_{T-1}}
\cdots
\circ\tau^{B[1]}_{k_1,\varepsilon_1}
\circ\tau^{B[0]}_{k_0,\varepsilon_0}
=
\mathrm{id}.
\end{align}
\end{defn}
Note that we do not assume the condition
\eqref{eq:bperiod1} here.

Each mutation in \eqref{eq:seq1}
is a  rational transformation,
while each transformation
in \eqref{eq:tropseq1} is  (the exponential form of) a linear transformation,
which is much simpler.
Surprisingly, the two  periodicities
in Definitions \ref{defn:period1} and \ref{defn:period2}
are equivalent.
The if-part of the following statement is
very nontrivial, and our proof is based on the recent result
by \cite{Cao17}.

\begin{prop}
\label{prop:trop1}
The sequence of mutations
\eqref{eq:seq1} is $\sigma$-periodic
if and only if
the sequence of  transformations
\eqref{eq:tropseq1} is $\sigma$-periodic.
\end{prop}
\begin{proof}
First, we note that,
by \cite[Proposition 1.3]{Nakanishi11a} 
and Theorem \ref{thm:sign1},
the sequence of  transformations
\eqref{eq:tropseq1}
is the exponential form of
 the transformations
of the corresponding $c$-vectors $c_i[s]=(c_{ji}[s])_{j=1}^n$ and $g$-vectors
$g_i[s]=(g_{ji}[s])_{j=1}^n$
along the sequence \eqref{eq:seq1},
where  $g$-vectors are defined in \cite[Section 6]{Fomin07}.
Thus, the $\sigma$-periodicity of the sequence 
\eqref{eq:tropseq1} is equivalent to the $\sigma$-periodicity of
 $c$- and $g$-vectors,
 i.e.,
  \begin{align}
\label{eq:cg1}
{c}_{ij}[T]={c}_{i\sigma(j)}[0]=\delta_{i\sigma(j)},
\quad
{g}_{ij}[T]={g}_{i\sigma(j)}[0]=\delta_{i\sigma(j)}.
\end{align}
\par
(Only-if-part.)
Assume that 
the sequence 
\eqref{eq:seq1} is $\sigma$-periodic.
The $\sigma$-periodicity of $c$-vectors directly follows
from
the $\sigma$-periodicity
of $y$-variables
\eqref{eq:yperiod1} by applying the tropicalization map
in \eqref{eq:tropmap1}.
Then,  the $\sigma$-periodicity of $g$-vectors follows
from the duality of $c$- and $g$-vectors in
\cite[Eq. (3.11)]{Nakanishi11a}.
\par
(If-part.) 
Assume that 
the sequence 
\eqref{eq:tropseq1} is $\sigma$-periodic.
The $\sigma$-periodicity of $c$-vectors  implies
the $\sigma$-periodicity of exchange matrices $B[s]$
thanks to  \cite[Eq.~(2.9)]{Nakanishi11a}.
Furthermore,
let 
$F_i[s]$ be the
 $F$-polynomials
along the sequence \eqref{eq:seq1},
which are defined in
 \cite[Section 3]{Fomin07}.
 Then,
by \cite[Theorem 2.5]{Cao17},
the $\sigma$-periodicity of $c$-vectors implies
the $\sigma$-periodicity of $F$-polynomials,
i.e.,
  \begin{align}
\label{eq:F1}
F_{\sigma^{-1}(i)}[T]=F_{i}[0].
\end{align}
Then, the $\sigma$-periodicity of $x$-variables
(resp. $y$-variables)
follows form the formula in 
\cite[Corollary 6.3]{Fomin07}
(resp.
\cite[Proposition 3.13]{Fomin07}).
\end{proof}

%\begin{rem}
%\label{rem:period1}
%The proof of Proposition \ref{prop:trop1} shows that
%one can safely omit the condition \eqref{eq:bperiod1}
%in Definition \ref{defn:period1}.
%\end{rem}

\begin{rem}
\label{rem:assum1}
Note that the above proof of the only-if-part uses
only the assumption 
\eqref{eq:yperiod1}.
\end{rem}

\subsection{Periodicity of canonical variables and tropical periodicity}
Let us consider the counterparts of 
the two  periodicities
in Definitions \ref{defn:period1} and \ref{defn:period2}
for canonical variables.

Let us define a (left) action of a permutation $\sigma$
of $\{1,\dots,n\}$ on
$\mathbb{R}^{2n}$,
\begin{align}
\label{eq:sigma4}
&\begin{matrix}
\sigma: &\mathbb{R}^{2n}&
\rightarrow 
&
 \mathbb{R}^{2n}\\
&
(u,p) &
\mapsto
&
(u',p'),
\end{matrix}\\
\label{eq:sigma5}
&\qquad\qquad
u'_i =
u_{\sigma^{-1}(i)},
\\
\label{eq:sigma6}
&\qquad\qquad
p'_i =
p_{\sigma^{-1}(i)}.
\end{align}
Since the map $\sigma$
is a canonical transformation,
we may regard it as a change of canonical coordinates on the phase space $M$.

Let $\varepsilon_0,\dots,\varepsilon_{T-1}$
continue to be the tropical sign sequence of \eqref{eq:seq1}.
Let  $(u[0],p[0])$ be an arbitrary point in $\mathbb{R}^{2n}$.
In parallel to the sequence of mutations \eqref{eq:seq1},
we consider a sequence of signed mutations
on $\mathbb{R}^{2n}$,
\begin{align}
\label{eq:seq2}
\begin{split}
(u[0],p[0])
\buildrel \mu^{B[0]}_{k_0,\varepsilon_0} \over \mapsto
(u[1],p[1])
\buildrel \mu^{B[1]}_{k_1,\varepsilon_1} \over \mapsto
\cdots
\buildrel \mu^{B[T-1]}_{k_{T-1},\varepsilon_{T-1}}
 \over \mapsto
(u[T],p[T]),
\end{split}
\end{align}
where
\begin{align}
\label{eq:decom3}
 \mu^{B[s]}_{k_s,\varepsilon_s} =
 \tau^{B[s]}_{k_s,\varepsilon_s} \circ
 \rho^{B[s]}_{k_s,\varepsilon_s},
 \quad
 (s=0,\dots,T-1),
\end{align}
and $ \rho^{B[s]}_{k_s,\varepsilon_s}$ are defined by \eqref{eq:sm1} under the following assumption:
\begin{ass}
We employ a common skew-symmetrizer $D$ 
of $B[s]$'s to define
$ \rho^{B[s]}_{k_s,\varepsilon_s}$'s.
\end{ass}

\begin{defn}
\label{defn:period3}
Let $\sigma$ be a permutation of $\{1,\dots,n\}$.
We say that a sequence of signed mutations
\eqref{eq:seq2} is {\em $\sigma$-periodic}
if the following conditions hold
 for any  initial point $(u[0],p[0])\in \mathbb{R}^{2n}$
 and  for any $1\leq i\leq n$:
\begin{align}
\label{eq:bperiod2}
b_{\sigma^{-1}(i)\sigma^{-1}(j)}[T]=b_{ij}[0],\\
\label{eq:uperiod1}
u_{\sigma^{-1}(i)}[T]=u_{i}[0],\\
\label{eq:pperiod1}
p_{\sigma^{-1}(i)}[T]=p_{i}[0].
\end{align}
\end{defn}

%\begin{rem}
%\label{rem:period2}
%According to Remark \ref{rem:period1} and the forthcoming  Proposition \ref{prop:},
%one can safely omit the condition \eqref{eq:bperiod2}
%in Definition \ref{defn:period1}.
%\end{rem}

The conditions 
\eqref{eq:uperiod1} and \eqref{eq:pperiod1} are also expressed as
the following equality 
on $\mathbb{R}^{2n}$:
\begin{align}
\label{eq:period3}
\sigma
\circ\mu^{B[T-1]}_{k_{T-1},\varepsilon_{T-1}}
\cdots
\circ\mu^{B[1]}_{k_1,\varepsilon_1}
\circ\mu^{B[0]}_{k_0,\varepsilon_0}
=
\mathrm{id}.
\end{align}

\begin{prop}
\label{prop:period2}
The sequence of mutations
\eqref{eq:seq1} is $\sigma$-periodic
if and only if
the sequence of signed mutations
\eqref{eq:seq2}  is  $\sigma$-periodic.
\end{prop}

The proof is a little lengthy, and it will be given in
Section \ref{subsec:pr1}.

\begin{ex}
\label{ex:inv1}
The inversion relation \eqref{eq:inv1}
is the simplest example of a $\sigma$-periodic sequence 
of mutations
\eqref{eq:seq1} with $T=2$, $k_1=k_2=k$, $\sigma=\mathrm{id}$,
and the tropical sign sequence is $\varepsilon_1=+$, $\varepsilon_2=-$.
The corresponding $\sigma$-periodic sequence of signed mutations 
is the inversion relation \eqref{eq:inv3} with $\varepsilon=+$.
\end{ex}

Next, let us consider the counterpart of the sequence \eqref{eq:tropseq1}
for canonical variables.
For the sequence \eqref{eq:seq2},
we introduce the following sequence of transformations,
\begin{align}
\label{eq:tropseq2}
\begin{split}
\mathbb{R}^{2n}
\buildrel \tau^{B[0]}_{k_0,\varepsilon_0} \over \rightarrow
\mathbb{R}^{2n}
\buildrel \tau^{B[1]}_{k_{1},\varepsilon_{1}}
\over\rightarrow
\cdots
\buildrel \tau^{B[T-1]}_{k_{T-1},\varepsilon_{T-1}}\over \rightarrow
\mathbb{R}^{2n},
\end{split}
\end{align}
where $\tau^B_{k,\varepsilon}$ is the one
  in \eqref{eq:tau8}.

\begin{defn}
\label{defn:period4}
Let $\sigma$ be a permutation of $\{1,\dots,n\}$.
We say that a sequence of  transformations
\eqref{eq:tropseq2} is {\em $\sigma$-periodic}
if the following equality holds
 on $\mathbb{R}^{2n}$:
\begin{align}
\label{eq:period2}
\sigma
\circ\tau^{B[T-1]}_{k_{T-1},\varepsilon_{T-1}}
\cdots
\circ\tau^{B[1]}_{k_1,\varepsilon_1}
\circ\tau^{B[0]}_{k_0,\varepsilon_0}
=
\mathrm{id}.
\end{align}
\end{defn}

\begin{prop}
\label{prop:trop2}
The sequence of transformations \eqref{eq:tropseq1} is $\sigma$-periodic
if and only if 
the sequence of transformations \eqref{eq:tropseq2} is $\sigma$-periodic.
\end{prop}
\begin{proof}
(Only-if-part.)
Assume that 
the sequence 
\eqref{eq:tropseq1} is $\sigma$-periodic.
Since \eqref{eq:tau1} is a log-version of \eqref{eq:trop1},
the $\sigma$-periodicity of $u$-variables follows from the
$\sigma$-periodicity of $g$-vectors in Proposition
\ref{prop:trop1}.
Similarly, 
since \eqref{eq:tau4} is a log-version of \eqref{eq:trop2},
the $\sigma$-periodicity of variables $(d_ip_i)_{i=1}^n$ follows from the
$\sigma$-periodicity of $c$-vectors in Proposition
\ref{prop:trop1}.
Then,  the $\sigma$-periodicity of $p$-variables follows
from this using Proposition \ref{prop:d1}.
\par
(If-part.) One can easily convert the above argument.
\end{proof}

Combining Propositions \ref{prop:trop1}, and \ref{prop:period2},
\ref{prop:trop2}, we reach  the following conclusion.

\begin{thm}
The following four conditions are  equivalent to each other:
\begin{itemize}
\item[(a).] The sequence of mutations
\eqref{eq:seq1} is $\sigma$-periodic.
\item[(b).] The sequence of transformations \eqref{eq:tropseq1} is $\sigma$-periodic.
\item[(c).] The sequence of signed mutations
\eqref{eq:seq2}  is  $\sigma$-periodic.
\item[(d).] The sequence of transformations \eqref{eq:tropseq2} is $\sigma$-periodic.
\end{itemize}
\end{thm}
\begin{proof}
We have (a) $\Longleftrightarrow$ (b) by 
Proposition \ref{prop:trop1},
(a) $\Longleftrightarrow$ (c) by 
Proposition \ref{prop:period2},
and
 (b) $\Longleftrightarrow$ (d) by 
Proposition \ref{prop:trop2}.
\end{proof}

%
%We have a  conjecture 
%which states that the tropical part determines the nontropical part.
%This is parallel to the one for the seed mutation.
%See, e.g., \cite[Theorem 2.6]{Nakanishi10c} and the remark after it.
%
%\begin{conj}
%\label{conj:period1}
%The sequence of signed mutations
%\eqref{eq:seq2} is $\sigma$-periodic
%if and only if the tropical periodicity 
%\eqref{eq:period2} holds.
%\end{conj}

\subsection{Proof of Proposition \ref{prop:period2}}
\label{subsec:pr1}

\subsubsection{If-part}
Let us prove the if-part of the proposition, which is easier.
Suppose that
the sequence of signed mutations
\eqref{eq:seq2}  is  $\sigma$-periodic.
By the conditions \eqref{eq:bperiod2}--\eqref{eq:pperiod1}
and Proposition \ref{prop:d1},
the $y$-variables defined by \eqref{eq:y1} 
satisfy the desired $\sigma$-periodicity 
\eqref{eq:yperiod1}  in $\mathbb{R}_+$.
Furthermore,
the initial $y$-variables
$y_1$, \dots, $y_n$
are algebraically independent
in $\mathbb{R}_+$
for a generic choice of  the initial point $(u[0],p[0])$.
  Therefore, the $\sigma$-periodicity \eqref{eq:yperiod1}  holds
in $\mathbb{Q}_+(y)$.

To prove the periodicity of $x$-variables,
we make use of Proposition \ref{prop:trop1}.
As noted in Remark \ref{rem:assum1},
from  the $\sigma$-periodicity \eqref{eq:yperiod1}
of the $y$-variables for the sequence \eqref{eq:seq1},
the $\sigma$-periodicity of 
the sequence of transformations \eqref{eq:tropseq1} holds.
Then, by the if-part of Proposition \ref{prop:trop1},
the $\sigma$-periodicity of the sequence
of mutations \eqref{eq:seq1} holds.

In the rest of this subsection,
we prove the only-if-part of the proposition.

\subsubsection{Hamiltonian point of view}
In our proof,
keeping the Hamiltonian point of view in mind is
very useful.
To make the presentation simple,
we consider the case $T=2$ in \eqref{eq:seq1}.
Although this is a toy example, it fully contains the idea of the proof for the general case.
%Nevertheless,
%the idea of the proof for a general case is fully contained.
To lighten the notation, 
let us abbreviate the flow in the left hand side of the sequence \eqref{eq:period3} as
\begin{align}
\label{eq:seq3}
(u,p)
\buildrel \mu \over \mapsto
(u',p')
\buildrel \mu' \over \mapsto
(u'',p''),
\buildrel \sigma \over \mapsto
(u''',p''').
\end{align}
Using the decomposition \eqref{eq:decom3} with a similar abbreviation,
we write it in the following way:
\begin{align}
\label{eq:flow1}
\xymatrix{
& (\tilde{u}',\tilde{p}')\ar@{|->}[r]^{\tau'}&(u'',p'')\ar@{|->}[r]^{\sigma}&
(u''',p''')
\\
(\tilde{u},\tilde{p})\ar@{|->}[r]^{\tau}&(u',p')
\ar@{|->}[u]^{\rho'}&\\
(u,p) \ar@{|->}[u]^{\rho} &
}
\end{align}
From the Hamiltonian point of view,
 the vertical maps $\rho$ and $\rho'$
 are Hamiltonian flows from $t=0$ to 1 and from $t=1$ to 2
  in the phase space $M$, respectively,
 while the horizontal maps $\tau$, $\tau'$, $\sigma$
 are changes of canonical coordinates of  $M$ so that points  do not move in  $M$.
Let us gather the piecewise Hamiltonian flow from $t=0$ to 2 in the initial chart.
 This can be done by the pull-back along the horizontal arrows
 as follows:
 \begin{align}
\label{eq:flow2}
\xymatrix{
(\tilde{\tilde{u}},\tilde{\tilde{p}})\ar@{|->}[r]^{\tau}& (\tilde{u}',\tilde{p}')\ar@{|->}[r]^{\tau'}&(u'',p'')\ar@{|->}[r]^{\sigma}&
(u''',p''')
\\
(\tilde{u},\tilde{p})\ar@{|->}[r]^{\tau}\ar@{|-->}[u]&(u',p')
\ar@{|->}[u]^{\rho'}&\\
(u,p) \ar@{|->}[u]^{\rho} &
}
\end{align}
 
The $\sigma$-periodicity  \eqref{eq:period3},
which we are going to show, states that
the points $(u,p)$ and $(u''',p''')$ coincide, but this coincidence happens
in different charts.
On the other hand, the tropical periodicity of
\eqref{eq:period2}
guaranteed by
Proposition \ref{prop:trop2} means that
\begin{align}
\sigma\circ \tau' \circ \tau = \mathrm{id}.
\end{align}
Thus, we have
\begin{align}
\label{eq:pb1}
(\tilde{\tilde{u}},\tilde{\tilde{p}})=(u''',p''').
\end{align}
Therefore, the $\sigma$-periodicity  \eqref{eq:period3}
is equivalent to the equality
in the initial chart,
\begin{align}
(\tilde{\tilde{u}},\tilde{\tilde{p}})=(u,p).
\end{align}
In other words, {\em the flow is periodic in the phase space $M$}.
We will show this separately for $u$- and $p$-variables.

\subsubsection{Periodicity of $p$-variables}
We start with $p$-variables.
Consider
\begin{align}
\begin{split}
\Delta p :=&\ \tilde{\tilde{p}} - p
=(\tilde{\tilde{p}} -\tilde{p})+(\tilde{p}- p)
\\
 =&\ \tau^{-1}(\tilde{p}'-p') + (\tilde{p}-p).
 \end{split}
\end{align}
By \eqref{eq:Hp2}, we have
\begin{align}
\label{eq:du1}
\tilde{p}_i-p_i&=-\frac{1}{2d_i}b_{ki} \log(1+y_k^{\varepsilon}),
\\
\label{eq:du2}
\tilde{p}'_i-p'_i&=-\frac{1}{2d_i}b'_{k'i} \log(1+y'_{k'}{}^{\varepsilon'}),
\end{align}
where we keep the same system of abbreviation.
Recall that $y$-variables here  are defined by
\begin{align}
\label{eq:yd1}
y_i =e^{d_i p_i+ w_i},
\quad
y'_i =e^{d_i p'_i+ w'_i},
\end{align}
and they obey the mutation rule \eqref{eq:My3}.
In particular, it is uniquely determined by the initial $y$-variables $y_i$.

Our goal is to show that $\Delta p=0$.
For this purpose, we compare the above flow of $p$-variables with
the (logarithm of) $y$-variables in the sequence of  mutations
\eqref{eq:seq1}.
In the same spirit of \eqref{eq:flow2} we write a diagram
for the sequence \eqref{eq:seq1},
 \begin{align}
\label{eq:flow3}
\xymatrix{
(\tilde{\tilde{x}},\tilde{\tilde{y}})\ar@{|->}[r]^{\tau}& (\tilde{x}',\tilde{y}')\ar@{|->}[r]^{\tau'}&(x'',y'')\ar@{|->}[r]^{\sigma}&
(x''',y''')
\\
(\tilde{x},\tilde{y})\ar@{|->}[r]^{\tau}\ar@{|-->}[u]&(x',y')
\ar@{|->}[u]^{\rho'}&\\
(x,y) \ar@{|->}[u]^{\rho} &
}
\end{align}
Let us set $v_i=\log y_i /d_i$, $v'_i=\log y'_i / d_i$, and so on.
%{\bf
Here, $\log y_i $ ($y_i \in \mathbb{Q}_+(y)$) is a formal notation such that
the multiplication in $\mathbb{Q}_+(y)$ is written additively.
In this notation the linear aspect of the transformation $\tau_{k,\varepsilon}^{B}$ in
\eqref{eq:tau7} is more transparent.
%}
Note that we have $v'''_i=v''_{\sigma^{-1}(i)}$ thanks to
 Proposition \ref{prop:d1}.
Then, we have
\begin{align}
\begin{split}
\Delta v :=&\ \tilde{\tilde{v}} - v
=(\tilde{\tilde{v}} -\tilde{v})+(\tilde{v}- v)
\\
 =&\ \tau^{-1}(\tilde{v}'-v') + (\tilde{v}-v).
 \end{split}
\end{align}
By \eqref{eq:auto3}, we have
\begin{align}
\label{eq:dv1}
\tilde{v}_i-v_i&=-\frac{1}{d_i}b_{ki} \log(1+y_k^{\varepsilon}),
\\
\label{eq:dv2}
\tilde{v}'_i-v'_i&=-\frac{1}{d_i}b'_{k'i} \log(1+y'_{k'}{}^{\varepsilon'}).
\end{align}
Let us compare them
with  \eqref{eq:du1} and \eqref{eq:du2}.
Note that
the $y$-variables here are elements in $\mathbb{Q}_+(y)$,
and they are different from the ones for   \eqref{eq:du1} and \eqref{eq:du2}.
However, they mutate by the  rule \eqref{eq:ymut1},
which is the same rule as for the $y$-variables in  \eqref{eq:flow3}.
Moreover,  the initial $y$-variables $y_i$ in \eqref{eq:seq1}  are the generators
of $\mathbb{Q}_+(y)$,
which  are formal (algebraically independent) variables.
Therefore, one can specialize them arbitrarily in $\mathbb{R}_+$,
so that they exactly match the initial $y$-variables for \eqref{eq:flow3}.
Under this specialization, we
have
\begin{align}
\label{eq:pv1}
\Delta p =\frac{1}{2} \Delta v,
\end{align}
where  we also  used the fact that $\tau$ is a linear transformation.
On the other hand, by the periodicity assumption
we have $ \Delta v=0$.
Therefore, $ \Delta p=0$ holds.

\subsubsection{Periodicity of $u$-variables}
%One can repeat the same argument of
%comparing the flows of the variables $d_ip_i$  
%in \eqref{eq:flow2} and the (logarithm of)
%$y$-variables in \eqref{eq:flow3}.
%The only subtlety here is that  when we show that
%$d_i{\tilde{\tilde{p_i}}}=d_i p'''_i$,
%which is parallel to \eqref{eq:pb1},
%we need  the equality
%$ d_i p_i''' = d_{\sigma^{-1}(i)}p_{\sigma^{-1}(i)}$.
%But, this is guaranteed by
%Proposition \ref{prop:d1}.
%Then, from the periodicity of $d_ip_i$,
%we obtain the periodicity of $p$-variables.

Due to the relation \eqref{eq:y1},
the periodicities of  $y$- and $p$-variables imply
the same periodicity of  $w$-variables.
Therefore, when the matrix $B[0]$ is {\em invertible},
the periodicity of $u$-variables immediately follows.
This reasoning, however, is not applicable when the matrix $B[0]$ is
not invertible.
Therefore, we prove the claim directly by comparing the mutations
of $u$- and $x$-variables in a similar way to the previous case.
The proof is parallel, but it requires some extra argument.

Consider
\begin{align}
\begin{split}
\Delta u :=&\ \tilde{\tilde{u}} - u
=(\tilde{\tilde{u}} -\tilde{u})+(\tilde{u}- u)
\\
 =&\ \tau^{-1}(\tilde{u}'-u') + (\tilde{u}-u),
 \end{split}
\end{align}
where, by \eqref{eq:Hu2},
\begin{align}
\label{eq:du3}
\tilde{u}_i-u_i&=-\frac{1}{2}\delta_{ki} \log(1+y_k^{\varepsilon}),
\\
\label{eq:du4}
\tilde{u}'_i-u'_i&=-\frac{1}{2}\delta_{k'i} \log(1+y'_{k'}{}^{\varepsilon'})
\end{align}
for the same $y_i$ and $y'_i$ in \eqref{eq:yd1}.

Let us compare the  flow of $u$-variables with
the (logarithm of) $x$-variables in the sequence of  mutations
\eqref{eq:flow3}.
Again, let us introduce formal logarithms,
 $z_i=\log x_i$, $z'_i=\log x'_i$, etc.,
 to write
the multiplication in $\mathcal{F}_{\mathbb{Q}_+(y)}$ in the additive way.
Then, we have
\begin{align}
\label{eq:dz0}
\begin{split}
\Delta z :=&\ \tilde{\tilde{z}} - z
=(\tilde{\tilde{z}} -\tilde{z})+(\tilde{z}- z)
\\
 =&\ \tau^{-1}(\tilde{z}'-z') + (\tilde{z}-z),
 \end{split}
\end{align}
where, by \eqref{eq:auto2},
\begin{align}
\label{eq:dz1}
\tilde{z}_i-z_i&=-\delta_{ki} \log(1+\hat{y}_k^{\varepsilon})
+\delta_{ki} \log(1\oplus {y}_k^{\varepsilon}),
\\
\label{eq:dz2}
\tilde{z}'_i-z'_i&=-\delta_{k'i} \log(1+\hat{y}'_{k'}{}^{\varepsilon'})
+\delta_{k'i} \log(1\oplus {y}'_{k'}{}^{\varepsilon'}).
\end{align}
It follows that we have the following expression of $\Delta z_i
:= \tilde{\tilde{z}}_i - z_i$,
\begin{align}
\label{eq:dz3}
\Delta z_i
= - \log f_i(\hat{y}) + \log f_i(y),
\end{align}
where $f_i(y)\in \mathbb{Q}_+(y)
$ is a rational function of 
the initial $y$-variables $y$,
and $f_i(\hat{y})\in \mathcal{F}_{\mathbb{Q}_+(y)}$ is
the one obtained from $f_i(y)$ by replacing
$y$ with the initial $\hat{y}$-variables $\hat{y}$
and also replacing
the addition in $\mathbb{Q}_+(y)$
with the one in $\mathcal{F}_{\mathbb{Q}_+(y)}$.
In the same way as \eqref{eq:pv1},
we have
\begin{align}
\label{eq:dz4}
\Delta u_i
:= \tilde{\tilde{u}}_i - u_i =- \frac{1}{2} \log f_i(y),
\end{align}
under some specialization of $y$ in the right hand side.

Now we set $\Delta z_i = 0$ for all $i$ by the assumption of the periodicity of $x$-variables.
This is equivalent to the equality $f_i(y)=f_i(\hat y)$ as elements of $\mathcal{F}_{\mathbb{Q}_+(y)}$.
We first claim that
 $f_i(y)$ is a {\em Laurent monomial  in $y$},
possibly with some coefficients in $\mathbb{Q}_+$.
In fact, if $f_i(y)$ is not a Laurent monomial,
 then it includes the addition in $\mathbb{Q}_+(y)$.
It follows that
 $ f_i(\hat{y})$ includes the addition in $\mathcal{F}_{\mathbb{Q}_+(y)}$.
 However,
 the addition in $\mathcal{F}_{\mathbb{Q}_+(y)}$ is an operation outside 
 of
  $\mathbb{Q}_+(y)$.
Thus, $ f_i(\hat{y})$ is not an  element in $\mathbb{Q}_+(y)
\subset \mathcal{F}_{\mathbb{Q}_+(y)}$.
 In particular, the equality $f_i(y)=f_i(\hat y)$ could never occur.
We next claim that actually we have
\begin{align}
f_i(y)=1.
\end{align}
To see it, we consider the limit $y_i\rightarrow 0$ for all $i$.
Then, thanks to the definition of the tropical sign,
we have $y_k^{\varepsilon}, y'_{k'}{}^{\varepsilon'}
\rightarrow 0$.
Thus, by \eqref{eq:dz0}--\eqref{eq:dz3},
we have $ \log f_i(y)\rightarrow 0$. Therefore, $f_i(y)$=1.
Thus, we conclude that $\Delta u_i =0$ by \eqref{eq:dz4}.

%Since the time one flow is a canonical transformation,
%we have
%\begin{align}
%\{ \tilde{\tilde{p}}_i,\tilde{\tilde{u}}_j\}=\delta_{ij},
%\quad
%\{ \tilde{\tilde{u}}_i,\tilde{\tilde{u}}_j\}=
%\{ \tilde{\tilde{p}}_i,\tilde{\tilde{p}}_j\}=0
%\end{align}
%for variables $\tilde{\tilde{u}}$ and $\tilde{\tilde{p}}$ defined by
%\eqref{eq:flow2}.
%Then, from the periodicity of $p$-variables,
%we obtain the periodicity of $u$-variables ${\tilde{\tilde{u_i}}}=u_i
%+c_i$ up to constants $c_i$.
%To determine the constants,
%we consider the limit $p_i\rightarrow -\infty$ for all $i$.
%Then, by \eqref{eq:yd1}, we have $y_i\rightarrow 0$;
%moreover, by the definition of the tropical sign,
%we have $y_k^{\varepsilon}, y'_{k'}{}^{\varepsilon'}
%\rightarrow 0$.
%Thus, $\tilde{u}-u, \tilde{u}'-u' \rightarrow 0$;
%therefore, $\Delta u =\tilde{\tilde{u}}-u \rightarrow 0$.
%Thus, $c_i=0$.

This completes the proof of Proposition \ref{prop:period2}.

\section{Dilogarithm identities and action integral}
\subsection{Dilogarithm identities}

The following theorem was proved in \cite{Nakanishi10c} by a
cluster algebraic method
with the help of the {\em constancy condition} from \cite{Frenkel95}. 
See also \cite{Nakanishi16}.
\begin{thm}
[{Dilogarithm identity \cite[Theorems 6.4 and 6.8]{Nakanishi10c}}]
\label{thm:DI}
Suppose that the sequence of mutations
\eqref{eq:seq1} is $\sigma$-periodic.
Let $\varepsilon_0,\dots,\varepsilon_{T-1}$
 be the tropical sign sequence of \eqref{eq:seq1}.
 Let $D=\mathrm{diag}(d_1,\dots,d_n)$ be any 
 skew-symmetrizer of the initial matrix $B$ in \eqref{eq:seq1}.
Then, the following identity of the Rogers dilogarithm $\tilde{L}(x)$ 
 in \eqref{eq:L4} holds:
\begin{align}
\label{eq:di1}
\sum_{s=0}^{T-1}
\frac{\varepsilon_s}{d_{k_s}}
\tilde{L}(y_{k_s}^{\varepsilon_s}[s])
= 0,
\end{align}
where $y_{k_s}[s]$ are evaluated by any
semifield homomorphism
\begin{align}
\label{eq:ev1}
\mathrm{ev}_y:
\mathbb{Q}_+(y)
\rightarrow 
\mathbb{R}_+.
\end{align}
\end{thm}

Below, we give an alternative proof of 
%a slightly weaker version of
 the theorem,
based on the Hamiltonian/Lagrangian picture presented in the current paper.
%As in Section \ref{subsec:pr1},
%we only consider the case $M=2$,
%since the idea of the proof for a general case is fully contained.

\subsection{Action integral}
%We consider the sequence \eqref{eq:seq2} with $M=2$,
%and  use the same system of  abbreviation in
%\eqref{eq:flow1}.
 
We consider
 the  sequence of signed mutations \eqref{eq:seq2}.

For each time span $[s,s+1]$ ($s=0,\dots,T-1$),
we have the Hamiltonian
\begin{align}
\label{eq:Hs1}
H[s]&=\frac{\varepsilon_s}{2d_k} \mathrm{Li}_2(-y_{k_s}[s]^{\varepsilon_s})
\quad
\text{for $[s,s+1]$}
\end{align}
in the $s$-th canonical coordinates $(u[s],p[s])$.
A Hamiltonian flow from time 0 to $T$ is a piecewise linear movement.
A  schematic diagram of a Hamiltonian flow is depicted in 
Figure \ref{fig:ham1} for $T=2$.
Let 
\begin{align}
\label{eq:lag4}
\mathscr{L}[s]
&=  \dot{u}_{k_s}[s] p_{k_s}[s] - H[s]
\quad
\text{for $ [s,s+1]$}
\end{align}
be the corresponding singular Lagrangian 
 from \eqref{eq:lag2}.

\begin{figure}
\begin{center}
\includegraphics{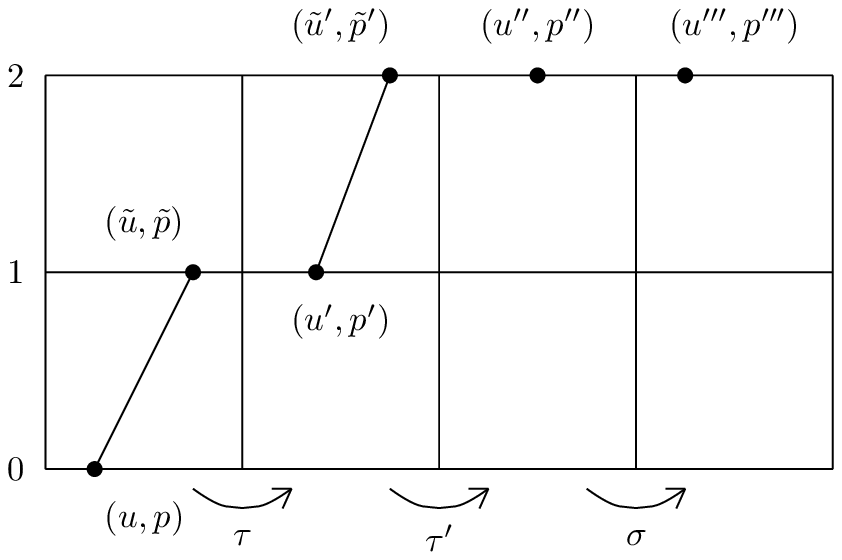}
\end{center}
\caption{Schematic diagram of a Hamiltonian flow for $T=2$,
where for simplicity
we use the same abbreviation as
in \eqref{eq:flow1}.
}
\label{fig:ham1}
\end{figure}

Let us consider the action integral ${S}$
along a Hamiltonian flow,
\begin{align}
\label{eq:Ihat1}
{S}&=\sum_{s=0}^{T-1} S[s],
\\
\label{eq:Ihat2}
S[s]&=
\int_s^{s+1}
\mathscr{L}[s](u(t),\dot{u}(t)) dt,
\end{align}
where $u(t)$ in \eqref{eq:Ihat2} is (the $u$-part of) a Hamiltonian flow in the $s$-th coordinates $(u[s],p[s])$.

%\begin{rem}
%Here we exclusively consider the action integrals
%${S}$
%along the Hamiltonian flow.
%\end{rem}

Here is the first key observation.
\begin{prop}
\label{prop:const1}
The value of the Lagrangian
\eqref{eq:lag4} 
along a Hamiltonian flow
is constant in $t$ in each time span $[s,s+1]$.
Thus, we have
\begin{align}
\label{eq:hi1}
{S}=
\sum_{s=0}^{T-1}
\mathscr{L}[s].
\end{align}
\end{prop}
\begin{proof}
Since the Hamiltonian is constant along the flow,
it is enough to show that 
the term  $\dot{u}_{k_s}[s] p_{k_s}[s] $ is constant.
This is true
since
we have $\ddot{u}_{k_s}[s]=\dot{p}_{k_s}[s]=0$
by Proposition \ref{prop:Hs1}.
\end{proof}

\subsection{Invariance of action integral}
Our next key observation
is as follows.

\begin{thm}
\label{thm:const1}
Suppose that the sequence of signed mutations
\eqref{eq:seq2} is $\sigma$-periodic.
Then, for any Hamiltonian flow,
\begin{align}
 {S}=0.
 \end{align}
 \end{thm}

The rest of this subsection is devoted to a proof of this theorem.
%As in Section \ref{subsec:pr1},
%we only consider the case $M=2$,
%since the idea of the proof for a general case is fully contained.
%We  use the same system of  abbreviation in
%\eqref{eq:flow1},
%such as $I[1]=I, I[2]=I'$, and so on.

First we show that the value ${S}={S}(u)$ is independent of 
a Hamiltonian flow $u(t)$,
using the standard variational calculus for a Lagrangian.
However, we have to be  careful because the Lagrangian here
is singular as discussed in Section \ref{subsec:lag1}.

For a given Hamiltonian flow $u(t)$, we consider
an infinitesimal variation $u(t)+\delta u(t)$, {\em which is also assumed to be 
a Hamiltonian flow}.
Let 
\begin{align}
\label{eq:DI1}
\delta S[s]&:=
S[s](u+\delta u) - S[s](u).
\end{align}
be
the variation of the action integral $S[s]$ in
the time span $[s,s+1]$.
We show that it is given by the boundary values of the time span as follows.
\begin{lem}
\label{lem:DI1}
\begin{align}
\label{eq:DI2}
\delta S[s]&
=
\sum_{i=1}^n
\tilde{p}_i[s] \delta \tilde{u}_i[s]
- 
\sum_{i=1}^n
p_i [s]\delta u_i [s],
\end{align}
where $(u[s],p[s])$ and $(\tilde{u}[s],\tilde{p}[s])$
are the points of the flow at $t=s$ and $s+1$, respectively,
in the $s$-th coordinates $(u[s],p[s])$.
\end{lem}
\begin{proof}
We fix the discrete time  $s$.
For simplicity, we suppress the index $s$
everywhere,
in particular we write $k$ instead of $k_s$.
Recall that, by Proposition \ref{prop:Hs1}, for $i\neq k$,
\begin{align}
\label{eq:ui1}
 \dot{u}_i=0.
\end{align}
Thus, we have, for $i\neq k$,
\begin{align}
\label{eq:dui1}
 \delta\dot{u}_i=0.
\end{align}
Therefore,
it is natural to separate the variation $\delta S=\delta S[s]$ into two parts:
\begin{align}
\delta S &=\delta S_1 + \delta S_2,\\
\delta S_1 &=
\int_s^{s+1}
\left(
\frac{\partial \mathscr{L}}{\partial u_k}
\delta u_k
+
\frac{\partial \mathscr{L}}{\partial \dot{u}_k}
\delta \dot{u}_k
\right)
dt,\\
\delta S_2 &=
\sum_{\scriptstyle i=1 \atop \scriptstyle i\neq k}^n
\int_s^{s+1}
\left(
\frac{\partial \mathscr{L}}{\partial u_i}
\delta u_i
\right)
dt.
\end{align}
Recall that,
by Proposition \ref{prop:le1} (1),
the Euler-Lagrange equation for $i=k$ 
follows from  the equations of motion.
By using it
and also \eqref{eq:leq1}, we have
\begin{align}
\label{eq:dI1}
\delta S_1 &=
\int_s^{s+1}
\left(
\frac{d}{dt}
\left(
\frac{\partial \mathscr{L}}{\partial \dot{u}_k}
\right)
\delta u_k
+
\frac{\partial \mathscr{L}}{\partial \dot{u}_k}
\delta \dot{u}_k
\right)
dt
=\left[\frac{\partial \mathscr{L}}{\partial \dot u_k} \delta u_k\right]_s^{s+1}
=
\left[
p_k \delta u_k
\right]_{s}^{s+1}.
\end{align}

On the other hand, for $i\neq k$,
by \eqref{eq:le4} we have an explicit expression for ${\partial \mathscr{L}}/{\partial u_i}$ and
\eqref{eq:ui1} shows $\delta u_i$ is independent of $t$.  Combining with \eqref{eq:Hdp2} gives
\begin{align}
\begin{split}
\label{eq:dI2}
\delta S_2
&=
\sum_{\scriptstyle i=1 \atop \scriptstyle i\neq k}^n
(- 1)\frac{b_{ki}}{2d_i}\log (1+y_k^{\varepsilon})
\delta u_i
=
\sum_{\scriptstyle i=1 \atop \scriptstyle i\neq k}^n
(\tilde{p}_i-p_i)
\delta u_i.
\end{split}
\end{align}
(Note that in \eqref{eq:dI2} we did not use the Euler-Lagrange 
equations for $i\neq k$,
which are not valid here as noted after Proposition \ref{prop:le1}.)
Recall from \eqref{eq:rho as flow} that $(\tilde u[s], \tilde p[s])$ is obtained from $(u[s],p[s])$ by the time one flow of the Hamiltonian $H[s]$.  Thus gathering \eqref{eq:dI1} with \eqref{eq:dI2} and noting that $\delta u_i=\delta\tilde u_i$ for $i\ne k$,
we obtain \eqref{eq:DI2}.
\end{proof}

Let $\sigma$ be the one
in \eqref{eq:sigma4},
and let us introduce
\begin{align}
(u[T+1],p[T+1])
:=
\sigma (u[T],p[T]).
\end{align}

\begin{lem}
We have the equalities
\label{lem:DI2}
\begin{align}
\label{eq:pus1}
\sum_{i=1}^n
{p}_i[s+1] \delta {u}_i[s+1]
&=
\sum_{i=1}^n
\tilde{p}_i[s] \delta \tilde{u}_i[s],
\quad
(s=0,\dots,T-1),
\\
\label{eq:pus2}
\sum_{i=1}^n
{p}_i[T+1] \delta {u}_i[T+1]
&=
\sum_{i=1}^n
{p}_i[T] \delta {u}_i[T].
\end{align}
\end{lem}
\begin{proof}
The first equality is due to Proposition \ref{prop:tau1} (1).
The second equality  is clear from the definition 
of $\sigma$.
\end{proof}

Consider the total variation
\begin{align}
\delta {S}
=\sum_{s=0}^{T-1} \delta S[s].
\end{align}
Combining  Lemmas \ref{lem:DI1} and \ref{lem:DI2},
we see that it is given by the boundary values,
\begin{align}
\delta {S} =
\sum_{i=1}^n
p_i[T+1] \delta u_i[T+1]
- 
\sum_{i=1}^n
p_i [0]\delta u_i[0].
\end{align}
On the other hand,
by the assumption
of the $\sigma$-periodicity of 
 the sequence of signed mutations
\eqref{eq:seq2},
 we have
\begin{align}
\label{eq:pus3}
{p}_i[T+1]={p}_i[0],
\quad
\delta {u}_i[T+1]=\delta {u}_i[0].
\end{align}
Therefore, we conclude that
\begin{align}
\delta {S} =0.
\end{align}
Since this holds for any infinitesimal variation of any flow $u(t)$,
$S$ is constant.

It remains to determine the constant value of $S$.
We evaluate it
in the limit $p_i[0]\rightarrow -\infty$ for all $i$.
Then, all initial $y$-variables $y_i[0]=\exp(d_i p_i[0] +w_i[0])$
go to $0$.
Accordingly, all $y_{k_s}[s]^{\varepsilon_s}$ ($s=0,\dots,T-1$) also go to $0$
due to the definition of the tropical sign $\varepsilon_s$.
So, by \eqref{eq:lag3}, the Lagrangians $\mathscr{L}[s]$
 go to $0$ as well.
Thus, we have $ S \rightarrow 0$.
Therefore, $ S = 0$.

This completes the proof of Theorem \ref{thm:const1}.

\begin{rem}
\label{rem:noether1}
The meaning of Theorem \ref{thm:const1} and its proof becomes
more transparent if we compare them with  {\em Noether's theorem}
(e.g., \cite[Theorem 1.3]{Takhtajan08}).

For simplicity, 
let us consider the variation of
a general  regular Lagrangian $\mathscr{L}$ under an infinitesimal transformation
$u_i \mapsto u_i + \epsilon a_i$ for some $i$,
where $\epsilon$ is an infinitesimal
and
$a_i$ is a function of $u$.
Then, by the Euler-Lagrange equation for $i$, we have
\begin{align}
\label{eq:del1}
\delta  \mathscr{L}
=
\frac{d}{dt}
\left(
\frac{\partial \mathscr{L}}{\partial \dot{u}_i}
\right)
\epsilon a_i
+
\frac{\partial \mathscr{L}}{\partial \dot{u}_i}
\epsilon \dot{a}_i
=
\epsilon
\frac{d}{dt}
\left(
p_i a_i
\right).
\end{align}
Thus, $\delta \mathscr{L}=0$,
i.e.,
the invariance of the Lagrangian in the order of $\epsilon$,
implies that  the generator $X=p_ia_i$ is an integral of motion;
that is Noether's theorem.
Moreover, we see in \eqref{eq:del1} that the converse is also true,
namely,
if $X=p_ia_i$ is an integral of motion,
then $\delta  \mathscr{L}=0$.

Next we consider a finite time analogue of the above.
Namely, we consider a variation of
the action integral $S$ under an infinitesimal transformation
of Hamiltonian flows
$u_i(t) \mapsto u_i(t) + \epsilon a_i(t)$ for some $i$,
where $\epsilon$ is an infinitesimal
and
$a_i(t)$ depends on $u(t)$.
Then, again by the Euler-Lagrange equation, we have
\begin{align}
\label{eq:del2}
\delta  S
=
\int_{t_0}^{t_1}
\left(
\frac{d}{dt}
\left(
\frac{\partial \mathscr{L}}{\partial \dot{u}_i}
\right)
\epsilon a_i
+
\frac{\partial \mathscr{L}}{\partial \dot{u}_i}
\epsilon \dot{a}_i
\right)
dt
=
\epsilon
\left[
p_i a_i
\right]_{t_0}^{t_1}.
\end{align}
Thus, 
$\delta S=0$,
i.e.,
the invariance of the action integral  in the order of $\epsilon$,
implies that the generator $X=p_ia_i$ is periodic at $t_0$ and $t_1$,
and {\em vice versa}.
\end{rem}

\subsection{Main results}

By combining Propositions
\ref{prop:Leg2} and \ref{prop:const1} with Theorem \ref{thm:const1},
we have the following theorem.
\begin{thm}
\label{thm:const2}
Suppose that the sequence of signed mutations
\eqref{eq:seq2} is $\sigma$-periodic.
Then, for the Hamiltonian flow of the Hamiltonian
 in \eqref{eq:Hs1}
with any initial point in the phase space $M$,
 we have
\begin{align}
\sum_{s=0}^{T-1}
\mathscr{L}[s]=0.
 \end{align}
 In particular, 
 if the initial point $(u[0],p[0])$ in \eqref{eq:seq2}
 satisfies the condition
 \begin{align}
 d_ip_i[0]
 = w_i[0],
 \quad
 (i=1,\dots,n),
 \end{align}
we have the following identity of the  Rogers dilogarithm $\tilde{L}(x)$
in \eqref{eq:L4}:
\begin{align}
\label{eq:di2}
\sum_{s=0}^{T-1}
\frac{\varepsilon_s}{d_{k_s}}
\tilde{L}(\hat{y}_{k_s}^{\varepsilon_s}[s])
= 0,
\end{align}
where
\begin{align}
\hat{y}_i[s]=e^{2w_i[s]}.
\end{align}
 \end{thm}
 
By  combining 
Proposition
 \ref{prop:period2} and Theorem \ref{thm:const2},
 we obtain a slightly different version of
 Theorem \ref{thm:DI}.
 \begin{thm}
\label{thm:const3}
Suppose that the sequence of  mutations
\eqref{eq:seq1} is $\sigma$-periodic.
Let $\varepsilon_0,\dots,\varepsilon_{T-1}$
 be the tropical sign sequence of \eqref{eq:seq1}.
 Let $D=\mathrm{diag}(d_1,\dots,d_n)$ be any 
 skew-symmetrizer of the initial matrix $B$ in \eqref{eq:seq1}.
We set the $y$-variables in \eqref{eq:seq1} to be trivial
by the specialization $\mathbb{Q}_+(y)\rightarrow \mathbf{1}$.
Let $\mathbb{Q}_+(x)\subset \mathcal{F}$ be a semifield
generated by the initial $x$-variables in \eqref{eq:seq1}.
Then, the following identity for the  Rogers dilogarithm $\tilde{L}(x)$
in \eqref{eq:L4} holds:
\begin{align}
\label{eq:di3}
\sum_{s=0}^{T-1}
\frac{\varepsilon_s}{d_{k_s}}
\tilde{L}(\hat{y}_{k_s}^{\varepsilon_s}[s])
= 0,
\end{align}
where 
\begin{align}
\label{eq:yhat10}
\hat{y}_i[s]=\prod_{j=1}^n x_j[s]^{b_{ji}[s]}
\end{align}
 are evaluated by any
semifield homomorphism
\begin{align}
\label{eq:ev2}
\mathrm{ev}_x:
\mathbb{Q}_+(x)
\rightarrow 
\mathbb{R}_+.
\end{align}
 \end{thm}
 
The only difference between 
Theorems \ref{thm:DI}
and  \ref{thm:const3} is the ranges of the initial $y$- and $\hat{y}$-variables therein
 under the evaluations \eqref{eq:ev1} and \eqref{eq:ev2}.
Namely,
 each of the initial $y$-variables in
  Theorem \ref{thm:DI}  independently takes {\em any\/} value in $\mathbb{R}_+$,
  since they are independent variables.
  Meanwhile, each of the  initial $\hat{y}$-variables in Theorem \ref{thm:const3},
  defined by \eqref{eq:yhat10} with $s=0$,
does so if and only if the initial matrix $B=B[0]$ is {\em invertible}.
Therefore, 
Theorem  \ref{thm:const3} is apparently weaker than
 Theorem \ref{thm:DI}.
 
 Nevertheless, we can show the following fact,
 which completes our derivation of Theorem \ref{thm:DI}.
 
 \begin{thm}
 \label{thm:equiv1}
 Theorem  \ref{thm:const3} implies 
  Theorem \ref{thm:DI}.
  Therefore,
  Theorems \ref{thm:DI}
and  \ref{thm:const3} are equivalent.
 \end{thm}

To show Theorem \ref{thm:equiv1},
we use the following notion.

\begin{defn}
Let $B$ and $\tilde{B}$ be skew-symmetrizable (integer) matrices
of size $n$ and $m$ ($n<m$), respectively.
We call  $\tilde{B}$  an {\em extension of $B$} if
$B$ is a principal submatrix of $\tilde{B}$.
If  an extension of $B$ is invertible,
then we call it  an {\em invertible extension of $B$}
\end{defn}

 For any skew-symmetrizable matrix $B$,
there is an invertible extension $\tilde{B}$ of $B$.
For example, if $D$ is a skew-symmetrizer of $B$,
then we have the following invertible extension of $B$:
\begin{align}
\tilde{B}=
\left(
\begin{array}{c|c}
B & -I\\
\hline
D & 0
\\
\end{array}
\right).
\end{align}

The following general fact on cluster algebras is  key  to proving Theorem \ref{thm:equiv1}.

\begin{thm}[{Extension Theorem (cf. \cite[Theorem 4.3]{Nakanishi10c})}]
\label{thm:ext1}
Let $B$ and $\tilde{B}$  be any skew-symmetrizable matrices 
of size $n$ and $m$, respectively, such
that $\tilde{B}$ is an extension of $B$.
Assume, for simplicity, that 
$B$ is the principal submatrix of $\tilde{B}$ for the first $n$ indices
$1,\dots,n$ of $\tilde{B}$.
Then, if the sequence \eqref{eq:seq1}
with the initial matrix $B=B[0]$
is $\sigma$-periodic,
the sequence \eqref{eq:seq1}
with the initial matrix being replaced with $\tilde{B}$
is also $\sigma$-periodic.
Here, a permutation $\sigma$ of $\{1,\dots,n\}$ is
naturally identified with a permutation of $\{1,\dots,m\}$
such that $\sigma(i)=i$ for $n+1\leq i \leq m$.
\end{thm}
 
 \begin{rem}
The above theorem is shown in  
\cite[Theorem 4.3]{Nakanishi10c} when $B$ is
skew-symmetric and $\sigma=\mathrm{id}$.
 \end{rem}
 
 \begin{proof}
Since the proof is parallel to the one in \cite[Theorem 4.3]{Nakanishi10c}),
 we only give a sketch of a proof.
We first show  the $\sigma$-periodicity of $c$-vectors.
Let  $c_i[s]=(c_{ji}[s])_{j=1}^m$ 
be the $c$-vectors for the 
$s$th seed  in the sequence
\eqref{eq:seq1} 
with the  initial  matrix $\tilde{B}$.
Then,
repeating the argument in the proof
of \cite[Theorem 4.3]{Nakanishi10c}),
 one can show the  $\sigma$-periodicity,
 \begin{align}
{c}_{ij}[T]={c}_{i\sigma(j)}[0]=\delta_{i\sigma(j)},
\end{align}
where we  use
the sign-coherence property
in Theorem \ref{thm:sign1},
the duality of $c$- and $g$-vectors in
\cite[Eq.(3.11)]{Nakanishi11a},
and
Proposition \ref{prop:d1}.
Then, by the proof of the if-part of the proof of Proposition \ref{prop:trop1},
 the $\sigma$-periodicity
 of seeds 
 is recovered from the $\sigma$-periodicity of
 $c$-vectors.
 \end{proof}

\begin{proof}[Proof of Theorem  \ref{thm:equiv1}]
As already mentioned, when the initial matrix $B=B[0]$ is invertible,
  Theorem  \ref{thm:const3} implies 
  Theorem \ref{thm:DI}.
  Suppose that $B$ is not invertible.
  Then, replace the initial matrix $B$ with
 any invertible extension $\tilde{B}$ of $B$.
  Thanks to Theorem \ref{thm:ext1},
  the sequence 
  \eqref{eq:seq1}
  enjoys the same $\sigma$-periodicity.  
A crucial observation is that
 the functional identity
\eqref{eq:di3} remains the same
 even if the initial matrix $B$
is replaced with $\tilde{B}$.
This is because,
 for $B$ and its extension $\tilde{B}$,
 the matrix mutation
in \eqref{eq:Bmut1},
and  also
the exchange relation of $y$-variables
in \eqref{eq:ymut1},
exactly coincide if all indices therein are restricted to the ones for $B$.
On the other hand, since $\tilde{B}$ is invertible,
each of the initial $y$-variables 
in the sequence
 \eqref{eq:di3} now independently takes any value in $\mathbb{R}_+$
 under the specialization \eqref{eq:ev2}.
 This is the desired result.
 \end{proof}
%  \begin{rem} 
%  Theorem  \ref{thm:const3} is apparently weaker than
%  Theorem \ref{thm:DI}
%  because
%   the semifield homomorphism $\mathbb{Q}_+(y) \rightarrow \mathbb{Q}_+(x)$,
% defined by $y_i\mapsto \hat{y}_i:=\prod_{j=1}^n x_j^{b_{ji}}$ is not an isomorphism,
% in general.
% However, 
% one can recover Theorem \ref{thm:DI} from
%  Theorem  \ref{thm:const3} in the following cases:
%  \par
% (1).  Suppose that the initial  matrix $B$ in \eqref{eq:seq1} is {\em invertible}.
% Then,  the above semifield homomorphism is an isomorphism.
% Thus,
%  Theorem \ref{thm:DI} for such $B$ follows
%  from   Theorem  \ref{thm:const3}.
%  \par
%  (2). Suppose that the initial matrix $B$ in \eqref{eq:seq1} 
%  is {\em skew-symmetric}.
%Then, the matrix $B$ can be extended
%   to
%  an {\em invertible} skew-symmetric matrix $\tilde{B}$.
%It is known that, if  the sequence \eqref{eq:seq1} is $\sigma$-periodic,
%   the sequence \eqref{eq:seq1} with $B$ being replaced by  $\tilde{B}$
%  is still $\sigma$-periodic \cite[Theorem 4.3]{Nakanishi10c}.
%By  (1),
%  we obtain
%    Theorem \ref{thm:DI} for $\tilde{B}$ from
%    Theorem \ref{thm:const3}.
%  Since the identities \eqref{eq:di1} 
%  for $B$ and $\tilde{B}$ are the same,
%  Theorem \ref{thm:DI} for $B$ also follows.
% \par 
%Moreover, we conjecture that Theorem 4.3 in \cite{Nakanishi10c} is
% extended to any {\em skew-symmetrizable}  matrix $B$
% and its  skew-symmetrizable extension $\tilde{B}$.
%Under the conjecture,
%one can fully recover
%   Theorem \ref{thm:DI} from
% Theorem \ref{thm:const3}.
%  \end{rem}

\bibliography{../../biblist/biblist.bib}
\end{document}